\RequirePackage{rotating}
\documentclass[11pt,francais]{amsart}
\usepackage{ae}
\usepackage{aeguill}
\usepackage[utf8]{inputenc}
\usepackage[T1]{fontenc}
\usepackage{amsmath,amsthm,amssymb}
\usepackage[francais]{babel}
\usepackage{graphicx}
\usepackage{latexsym}
\usepackage{CJK}
\usepackage[all]{xy}
\usepackage{amsmath}
\usepackage{cancel}
\usepackage{mathtools}
\usepackage{tikz}
\usepackage{slashbox}
\usepackage{pict2e}
\usetikzlibrary{chains}
\usepackage{rotating}
\usepackage[linktoc=all]{hyperref}
\hypersetup{colorlinks=true,linkcolor=blue,citecolor=red,urlcolor=blue,pdfborderstyle={/S/U/W1}}
\usepackage{changepage}
\usepackage{multirow}
\usepackage{geometry}
\usepackage{bigcenter}

\tikzset{node distance=2em, ch/.style={circle,draw,on chain,inner sep=2pt},chj/.style={ch,join},every path/.style={shorten >=4pt,shorten <=4pt},line width=1pt,baseline=-1ex}

\addto\captionsfrench{}

\def\R{\mathbb{R}}
\def\C{\mathbb{C}}
\def\Q{\mathbb{Q}}

\def\Z{\mathbb{Z}}
\def\N{\mathbb{N}}

\def\F{\mathbb{F}}

\usepackage{babel}
\usepackage{array}

\newtheorem{defi}{Définition}[subsection]
\newtheorem{theo}[defi]{Théorème}
\newtheorem{prop}[defi]{Proposition}
\newtheorem{propdef}[defi]{Proposition-Définition}

\newtheorem{lem}[defi]{Lemme}

\newtheorem{cor}[defi]{Corollaire}

\usepackage[all]{xy}
\title[Calcul d'opérateurs de Hecke]{Calcul des opérateurs de Hecke sur les classes d'isomorphisme de réseaux pairs de déterminant $2$ en dimension $23$ et $25$.}
\author{Thomas Mégarbané}
\address{CMLS, {\'E}cole polytechnique, CNRS, Universit{\'e} Paris-Saclay, 91128 Palaiseau Cedex, France}
\email{thomas.megarbane@polytechnique.edu}
\begin{document}
\maketitle
\begin{abstract}
Dans cet article, nous calculons l'opérateur de Hecke $\mathrm{T}_2$ associé aux $2$-voisins de Kneser défini sur les classes d'isomorphisme des réseaux pairs de déterminant $2$ en dimension $23$ et $25$. Grâce aux résultats de \cite{Meg}, on en déduit l'expression de nombreux autres opérateurs de Hecke. Ceci nous permet de déterminer pour tout $p$ premier le graphe de Kneser associé aux $p$-voisins des réseaux de dimension $23$ ou $25$. Nos résultats permettent aussi d'améliorer la Conjecture de Harder, et de démontrer de nombreuses autres congruences faisant intervenir les paramètres de Satake des représentations automorphes des groupes linéaires découvertes par Chenevier et Renard. 
\end{abstract}
\section{Introduction.}
Fixons $n\equiv 0,\pm 1\ \mathrm{mod}\ 8$ un entier strictement positif, et considérons un espace euclidien $V$ de dimension $n$. On définit l'ensemble $\mathcal{L}_n$ des réseaux pairs $L\subset V$ tels que $\mathrm{det}(L)=1$ si $n$ est pair, et $\mathrm{det}(L)=2$ sinon. L'ensemble $\mathcal{L}_n$ est muni d'une action du groupe orthogonal euclidien $\mathrm{O}(V)\simeq \mathrm{O}_n(\R)$, et on note $X_n=\mathrm{O}(V) \setminus \mathcal{L}_n$.\\

Suivant Kneser, si l'on se donne $A$ un groupe abélien fini, on dit que les réseaux $L_1,L_2\in \mathcal{L}_n$ sont des $A$-voisins si :
$$L_1/(L_1\cap L_2)\simeq L_2/(L_1\cap L_2)\simeq A.$$

On parle plus simplement des $d$-voisins lorsque $A=\Z/d\Z$ : c'est le cas qui nous intéresse le plus. Une fois un réseau $L\in \mathcal{L}_n$ donné, il est facile de construire tous ses $d$-voisins, comme rappelé à la proposition \ref{dvoisin}.

Cette notion de $A$-voisin, et plus particulièrement celle de $p$-voisins (pour $p$ un nombre premier), nous permet de définir à $n$ fixé un endomorphisme $\mathrm{T}_p$ sur le $\Z$-module libre $\Z[X_n]$ engendré par $X_n$. On le définit par $\mathrm{T}_p(\overline{L})=\sum \overline{L'}$, la somme portant sur les $p$-voisins $L'$ de $L$, et $\overline{L}$ (respectivement $\overline{L'}$) désignant la classe dans $X_n$ de $L$ (respectivement $L'$).\\

L'étude de l'endomorphisme $\mathrm{T}_p\in \mathrm{End}(\Z[X_n])$ passe par la compréhension de l'ensemble $X_n$.
 
Lorsque $n\leq 9$, on sait d'après Mordell (pour $n=8$) et par exemple d'après Conway-Sloane \cite{CS99} (pour $n\in \{1,7,9\}$) que $\vert X_n \vert =1$, et l'opérateur $\mathrm{T}_p$ n'est pas très pertinent.

Lorsque $n\in \{15,16,17\}$, les ensembles $X_n$ ont été déterminés par Witt (pour $n=16$) et Conway-Sloane (pour $n=15,17$). Suivant les résultats de Chenevier-Lannes \cite{CL}, l'opérateur $\mathrm{T}_p$ se déduit de l'étude des formes modulaires paraboliques pour $\mathrm{SL}_2(\Z)$. La connaissance de $\mathrm{T}_p$ est équivalente à la donnée, pour tous $L,L'\in \mathcal{L}_n$, du nombre de $p$-voisins de $L$ isomorphes à $L'$. Ces quantités font intervenir des polynômes en $p$ ainsi que le $p$-ème terme du $q$-développement des formes modulaires normalisées paraboliques pour $\mathrm{SL}_2(\Z)$ de poids $12$ ou de poids $16$. Pour une étude détaillée, nous renvoyons à \cite[ch. I, Théorème A]{CL} lorsque $n=16$, et à \cite[Annexe B, \S 5]{CL} lorsque $n=15,17$.

Lorsque $n\in \{ 23,24,25\}$, la classification des éléments de $X_n$ est le produit des travaux de Niemeier (pour $n=24$, ce qui donne aussi la classification pour $n=23$) et de Borcherds (pour $n=25$). On prendra bien garde au fait que $\vert X_{23} \vert =32$, $\vert X_{24} \vert =24$ et $\vert X_{25} \vert =121$, et il est facile de se tromper sur les indices qui interviennent dans la suite.

Si les ensembles $X_{23},X_{24}$ et $X_{25}
$ sont plus ou moins bien connus, il n'y a que pour $n=24$ que des opérateurs $\mathrm{T}_p\in \mathrm{End}(\Z[X_n])$ ont été déterminés. Le calcul de l'opérateur $\mathrm{T}_2$ sur $\Z[X_{24}]$ résulte des travaux de Borcherds \cite{Bor84} \cite{CS99}, repris ensuite par Nebe-Venkov dans \cite{NV}. L'étude faite par Chenevier-Lannes dans \cite{CL} repose sur la codiagonalisation des opérateurs $\mathrm{T}_p\in\mathrm{End}(\Z[X_{24}])$, et permet d'en déduire pour $p\leq 113$ l'opérateur $\mathrm{T}_p$ sur $\Z[X_{24}]$. Ils utilisent pour cela que les valeurs propres de l'opérateur $\mathrm{T}_2$ sont toutes distinctes, et la diagonalisation de $\mathrm{T}_2$ fournit une base de codiagonalisation pour tous les $\mathrm{T}_p$.

Le premier but de notre travail est de déterminer un maximum d'opérateurs $\mathrm{T}_p$ pour $n=23$ et $n=25$.\\

Notre point de départ est la détermination de l'opérateur $\mathrm{T}_2$ lorsque $n=23$ et $n=25$, ce qui fait l'objet du paragraphe \ref{3}.

Au paragraphe \ref{3.1}, on étudie les ensembles $X_{23}$ et $X_{25}$. On étudie le rôle fondamental que jouent les systèmes de racines des réseaux de $\mathcal{L}_{23}$ et $\mathcal{L}_{25}$ dans la compréhension de $\mathrm{X}_{23}$ et $X_{25}$, détaillé à la proposition \ref{RXn}, et certainement déjà connu de Borcherds : si $n=23$ ou $25$, deux réseaux $L_1,L_2\in \mathcal{L}_n$ sont isomorphes si, et seulement si, leurs systèmes de racines $R(L_1),R(L_2)$ sont isomorphes. On possède un résultat analogue lorsque $n=24$, qui se déduit des travaux de Niemeier \cite{Nie} et Venkov \cite{Ven}.

Au paragraphe \ref{3.2}, on explique comment déterminer la classe d'un réseau $L'\in \mathcal{L}_n$ dans $X_n$, où les données sont les suivantes : on possède un réseau $L\in \mathcal{L}_n$ (défini par une $\Z$-base), et $L'$ est un $2$-voisin de $L$ (déterminé suivant la construction de \cite[Annexe B, propositions 3.3 et 3.4]{CL} par un vecteur isotrope non-nul de $L/2L$). L'algorithme présenté dans ce paragraphe nous rend une $\Z$-base de $L'$, ainsi que sa classe d'isomorphisme dans $X_n$. 

Au paragraphe \ref{3.3}, on détaille l'algorithme permettant de calculer $\mathrm{T}_2$ sur $X_{23}$ et $X_{25}$. Il se déduit directement des paragraphes précédents : il suffit de parcourir, une fois donnés des réseaux $L_1,\dots,L_k\in \mathcal{L}_n$ d'images distinctes dans $X_n$ (avec $\vert X_n \vert =k$) tous leurs $2$-voisins, et d'en déterminer les classes d'isomorphisme. La connexité du graphe de Kneser $\mathrm{K}_n(2)$ facilite grandement notre tâche. Il suffit de considérer un élément quelconque de $\mathcal{L}_n$, et de parcourir ses $2$-voisins. En réitérant ce procédé aux $2$-voisins des réseaux ainsi construits, on arrive à parcourir tous les éléments de $X_n$. On construit ainsi une famille $(L_1,\dots,L_k)$ satisfaisant les conditions ci-dessus, à l'aide uniquement de la donnée d'un élément de $\mathcal{L}_n$ quelconque.

Au final, nous obtenons les matrices de $\mathrm{T}_2$ sur $\Z[X_{23}]$ et $\Z[X_{25}]$, exprimées dans les base de $\Z[X_{23}]$ et $\Z[X_{25}]$ correspondant respectivement à la numérotation des tables \ref{rX23} et \ref{rX25}. Ces matrices sont données dans \cite{MegTp}. Notons au passage que cette même méthode permettrait aussi de recalculer la matrice de $\mathrm{T}_2$ sur $\Z[X_{24}]$.\\

\`A n fixé, la codiagonalisation sur $\C$ des opérateurs de Hecke (et donc de leurs matrices associées) permet de décrire les matrices des opérateurs $\mathrm{T}_p$ pour $p$ suffisamment petit, ce que l'on présente au paragraphe \ref{4.1}.

La méthode qu'on utilise est la même que celle déjà utilisée par Chenevier-Lannes \cite{CL} en dimension $24$. L'opérateur $\mathrm{T}_2$ a ses valeurs propres deux à deux distinctes, et est connu explicitement. On possède ainsi une base de diagonalisation de $\mathrm{T}_2$, qui est aussi une base de codiagonalisation pour tous les opérateurs $\mathrm{T}_p$, vus comme des endomorphismes de $\C[X_n]$. Il suffit ensuite d'exprimer les valeurs propres associées à cette base de diagonalisation, ce qui se déduit des résultats de  Chenevier-Lannes \cite[Table C.7]{CL} et de Chenevier-Renard \cite[Appendix D]{CR}, et que l'on détaille aux propositions \ref{lambdai} et \ref{mui}.

Suivant les notations de \cite{CR} ou \cite{CL}, notons $\Pi_{\mathrm{alg}}^\bot (\mathrm{PGL}_m)$ l'ensemble des classes d'isomorphisme de représentations automorphes cuspidales autoduales de $\mathrm{GL}_m$ sur $\Q$, telles que $\pi_p$ est non ramifiée pour tout $p$, et que $\pi_\infty$ est algébrique régulière. Alors les valeurs propres de l'opérateur $\mathrm{T}_p$ sur $\Z[X_n]$ s'expriment grâce à la trace du $p$-ème paramètre de Satake d'éléments de $\Pi_{\mathrm{alg}}^\bot(\mathrm{PGL}_m)$, avec $m\in \{2,3,4\}$ pour $n=23$ et $m\in \{2,3,4,6 \}$ pour $n=25$. Pour $m=2$ ou $3$, ces quantités sont bien connues pour tout $p$ premier, et se déduisent des coefficients du $q$-développement des formes modulaires pour $\mathrm{SL}_2(\Z)$ de poids $\leq 23$. Pour $m=4$ et $n=23$, ces quantités sont bien connues pour $p\leq 113$, puisqu'elles s'expriment grâce aux coefficients $\tau_{j,k}(p)$ qui ont été calculés par Chenevier-Lannes \cite{CL}. Enfin, pour $m=4$ ou $6$, et $n=25$, ces quantités ont été calculées dans \cite{Meg} pour $p\leq 67$.

Nos résultats permettent ainsi d'expliciter de nombreux opérateurs $\mathrm{T}_p$ sur $\Z[X_n]$, et on a le théorème suivant :

\begin{theo} Pour $n=23$ et $p\leq 113$, ou $n=25$ et $p\leq 67$, l'endomorphisme $\mathrm{T}_p\in \mathrm{End}(\Z[X_n])$ est donné dans \cite{MegTp}.
\end{theo}

Une première application de nos résultats est la détermination, pour $n=23$ et $p\leq 113$, ou $n=25$ et $p\leq 67$, du graphe $\mathrm{K}_n(p)$, défini au paragraphe \ref{2.2}. Rappelons que le graphe de Kneser $\mathrm{K}_n(p)$ est défini comme le graphe dont les sommets sont les éléments de $X_n$, et dont les arêtes sont les $\{ \overline{L}_1,\overline{L}_2 \}$ pour $L_1,L_2\in \mathcal{L}_n$ des $p$-voisins. Au paragraphe \ref{4.1}, on démontre que l'on a le théorème suivant :

\begin{theo} Soit $p$ un nombre premier :

\begin{enumerate}
\item[$(i)$] Le graphe $\mathrm{K}_{23}(p)$ est complet si, et seulement si, $p\geq 23$.
\item[$(ii)$] Le graphe $\mathrm{K}_{25}(p)$ est complet si, et seulement si, $p\geq 67$.
\end{enumerate}
\end{theo}

Ainsi, nos résultats permettent de déterminer pour tout $p$ premier les graphes $\mathrm{K}_{23}(p)$ et $\mathrm{K}_{25}(p)$.\\

Une deuxième application de nos résultats provient de l'étude de la base de codiagonalisation des opérateurs de Hecke trouvée grâce aux vecteurs propres de $\mathrm{T}_2$, ce qui fait l'objet du paragraphe \ref{4.3}.

Une telle étude permet dans un premier temps de redémontrer la ``Conjecture de Harder" \cite{Har}, déjà démontrée dans \cite[Introduction, Théorème I]{CL} par une étude des opérateurs $\mathrm{T}_p$ sur $\Z[X_{24}]$. Mieux : on l'améliore ici sous la forme du théorème suivant, où les notations sont celles du paragraphe \ref{2.3} :
\begin{theo} Pour tout nombre premier $p$, on a la congruence :
\[D_{21,5}(p)\equiv D_{21}(p)+p^{13}+p^8\mathrm{\ mod\ }9840.\]

De plus, cette congruence est optimale, dans le sens où on ne peut pas remplacer $9840$ par un de ses multiples : on a $D_{21}(p)+p^{13}+p^8-D_{21,5}(p)= 9840$ pour $p=2$.
\end{theo}

Suivant la même méthode, on démontre de nombreuses autres congruences qui sont présentées en détail au paragraphe \ref{4.3}. De même que pour la congruence précédente, certaines des congruences exposées avaient déjà été démontrées dans \cite{CL}. La démonstration qu'on en fait ici est plus facile pour la raison suivante. Dans \cite{CL}, l'étude des valeurs propres de $\mathrm{T}_p$ permettait d'obtenir des ``multiplications par $(p+1)$" des congruences cherchées, et le fait de ``diviser par $(p+1)$" pose problème lorsque $(p+1)$ n'est pas premier au module de la congruence. Ici, on obtient directement les congruences cherchées, ou des ``multiplications par $p$" de ces congruences, et le fait de ``diviser par $p$" est beaucoup plus facile (car il suffit d'évaluer la congruence pour $p$ divisant le module de la congruence).

Au final, nous démontrons également le théorème suivant :

\begin{theo} Pour tout nombre premier $p$, les congruences suivantes sont vérifiées :

\begin{enumerate}
\item[$(i)$] $D_{19,7}(p) \equiv D_{19}(p)+p^6+p^{13}\mathrm{\ mod\ }8712$ ;
\item[$(ii)$] $D_{21,5}(p) \equiv D_{21}(p)+p^8+p^{13}\mathrm{\ mod\ }9840$ ;
\item[$(iii)$] $D_{21,9}(p) \equiv (1+p^6)\,D_{15}(p)\mathrm{\ mod\ }12696$ ;
\item[$(iv)$] $D_{21,9}(p) \equiv D_{21}(p)+p^6+p^{15}\mathrm{\ mod\ }31200$ ;
\item[$(v)$] $D_{21,13}(p) \equiv (1+p^4)\, D_{17}(p)\mathrm{\ mod\ }8736$ ;
\item[$(vi)$] $D_{21,13}(p) \equiv D_{21}(p)+p^4+p^{17}\mathrm{\ mod\ }10920$ ;
\item[$(vii)$] $D_{23,7}(p)\equiv (1+p^8)\, D_{15}(p)\mathrm{\ mod\ }8972$ ;
\item[$(viii)$] $D_{23,13,5}(p)\equiv D_{23,13}(p)+p^9+p^{14}\mathrm{\ mod\ }5472$ ;
\item[$(ix)$] $D_{23,15,7}(p)\equiv (1+p^4)\, D_{19}(p)+p^8+p^{15}\mathrm{\ mod\ }2184$ ;
\item[$(x)$] $D_{23,15,7}(p)\equiv D_{23,7}(p)+p^4\, D_{15}(p)\mathrm{\ mod\ }5856$ ;
\item[$(xi)$] $D_{23,17,9}(p)\equiv D_{23,9}(p)+p^3\, D_{17}(p)\mathrm{\ mod\ }2976$ ;
\item[$(xii)$] $D_{23,19,3}(p)\equiv (1+p^2)\, D_{21}(p)+p^{10}+p^{13}\mathrm{\ mod\ }7872$ ;
\item[$(xiii)$] $D_{23,19,11}(p)\equiv (1+p^2)\, D_{21}(p)+p^6+p^{17}\mathrm{\ mod\ }16224$.
\end{enumerate}

De plus, mis à part les points $(vi),(vii),(xi)$ et $(xiii)$, les congruences ci-dessus sont optimales, dans le sens où le module qui intervient ne peut pas être remplacé par un de ses multiples.
\end{theo}

Notons au passage que la congruence $(viii)$ avait déjà été conjecturée dans \cite[\S 6, Example 3]{BDM}. On la démontre sous la forme d'un résultat plus fort que dans \cite{BDM}, et on vérifie que ce résultat est optimal. \\

Enfin, nos résultats valident dans certains cas particuliers une conjecture de Gan-Gross-Prasad, exposée en conclusion de \cite[Classical groups, the local case]{GGP}. Cette dernière stipule que les paramètres standards des représentations $\pi,\pi'$ de $\mathrm{SO}_{m},\mathrm{SO}_{m-1}$ déterminent entièrement si, et seulement si, $\pi'$ est une restriction de $\pi$. On donne plus en détail au paragraphe \ref{4.4} ce critère sur les paramètres standards de $\pi$ et $\pi'$.

Nos résultats permettent de déterminer, lorsque $\pi\in\Pi_{\mathrm{disc}}(\mathrm{O}_{24})$ et $\pi'\in \Pi_{\mathrm{cusp}}(\mathrm{SO}_{23})$, ou lorsque $\pi'\in \Pi_{\mathrm{cusp}}(\mathrm{SO}_{25})$ et $\pi\in\Pi_{\mathrm{disc}}(\mathrm{O}_{24})$, si $\pi'$ est une restriction de $\pi$. Les paramètres standards de telles représentations ont été déterminés dans \cite[Table C.7]{CL} pour les éléments de $\Pi_{\mathrm{cusp}}(\mathrm{SO}_{23})$, \cite[Table C.5]{CL} pour les éléments de $\Pi_{\mathrm{disc}}(\mathrm{O}_{24})$ et \cite[Appendix D]{CR} pour les éléments de $\Pi_{\mathrm{cusp}}(\mathrm{SO}_{25})$.

Prasad et Chenevier avaient réalisé une inspection des paramètres standards des éléments de $\Pi_{\mathrm{cusp}}(\mathrm{SO}_{23})$, $\Pi_{\mathrm{disc}}(\mathrm{O}_{24})$ et $\Pi_{\mathrm{cusp}}(\mathrm{SO}_{25})$. Ils avaient alors remarqué que, en se donnant $\pi \in \Pi_{\mathrm{disc}}(\mathrm{O}_{24})$ (respectivement $\pi \in \Pi_{\mathrm{cusp}}(\mathrm{SO}_{25})$), le sous-ensemble $\Pi' \subset \Pi_{\mathrm{cusp}}(\mathrm{SO}_{23})$ (respectivement $\Pi' \subset \Pi_{\mathrm{disc}}(\mathrm{O}_{24})$) dont les éléments satisfont la conjecture de Gan-Gross-Prasad est non vide. Nos résultats vont dans le sens de cette constatation, puisque l'on a en fait l'égalité $\Pi'=\mathrm{Res}(\pi)$. On en déduit le théorème suivant :

\begin{theo} La conjecture de Gan-Gross-Prasad est bien vérifiée lorsque $\pi\in\Pi_{\mathrm{disc}}(\mathrm{O}_{24})$ et $\pi'\in \Pi_{\mathrm{cusp}}(\mathrm{SO}_{23})$, ou lorsque $\pi'\in \Pi_{\mathrm{cusp}}(\mathrm{SO}_{25})$ et $\pi\in\Pi_{\mathrm{disc}}(\mathrm{O}_{24})$.
\end{theo}

\bigbreak
\bigbreak
\bigbreak
Cet article a été écrit dans le cadre de ma thèse sous la direction de Gaëtan Chenevier, que je remercie pour les discussions utiles que nous avons pu avoir. Je remercie aussi Jean Lannes, qui a montré beaucoup d'intérêt pour mes résultats, et avec qui j'ai pu également beaucoup échanger.
\newpage
\tableofcontents
\newpage
\section{Résultats préliminaires et notations.}\label{2}

Dans toute la suite, on se place dans un espace euclidien $V$ de dimension $n$, muni de son produit scalaire $x\cdot y$, et on note $q:V\rightarrow \R,x\mapsto \frac{x\cdot x}{2}$ la forme quadratique associée. On considèrera souvent le cas où $V=\R^n$, muni de sa structure euclidienne, avec pour base canonique associée $(e_i)_{i\in \{1,\dots,n\}}$. On notera alors $(x_i)\cdot (y_i)=\sum_i x_iy_i$ le produit scalaire usuel.

\subsection{Les réseaux de $\R^n$.}\label{2.1}

\begin{defi}[Réseaux entiers et pairs] Soit $L\subset V$ un réseau. On dit que $L$ est entier si :
$$(\forall x,y\in L)\ x\cdot y \in \Z.$$

Si l'on se donne un réseau $L\subset V$ entier, il est dit pair si :
$$(\forall x\in L)\ x\cdot x\in 2 \Z.$$
\end{defi}

\begin{defi}[Dual et résidu d'un réseau] Soit $L\subset V$ un réseau. On définit $L^\sharp$ le dual de $L$ par :
$$L^\sharp = \{ y\in V \vert (\forall x\in L)\ y\cdot x \in \Z \}.$$

En particulier, $L$ est entier si, et seulement si, $L\subset L^\sharp$. Dans ce cas on définit le résidu de $L$ comme le quotient :
$$\mathrm{r\acute{e}s}\,  L = L^\sharp /L .$$

Ce quotient est muni d'une forme quadratique $\mathrm{r\acute{e}s}\,  L \rightarrow \Q/\Z$ définie par $x\mapsto q(x)\mathrm{\ mod\ }\Z$ appelée forme d'enlacement.
\end{defi}

\begin{defi}[Déterminant d'un réseau] Soit $L$ un réseau entier. On note $\mathrm{det}(L)$ son déterminant, qui est encore le déterminant de la matrice de Gram d'une base quelconque de $L$. On a la relation bien connue :
$$\mathrm{det}(L) = \vert \mathrm{r\acute{e}s}\, L \vert . $$
\end{defi}

\begin{defi}[Racines d'un réseau]
Soit $L\subset V$ un réseau entier. On définit le système de racines de $L$ comme l'ensemble $R(L)$ (qui est fini, et éventuellement vide) :
$$ R(L) = \{x \in L \vert x\cdot x = 2\}. $$

C'est un système de racines du $\R$-espace vectoriel qu'il engendre au sens de \cite[ch. VI, \S 1.1, définition 1]{Bo}, ce qui justifie la terminologie (c'est même un système de racines de type ADE).
\end{defi}

\begin{propdef}[Racines positives et racines simples] Soient $R$ un système de racines de $V$, et $D$ un demi-espace. On suppose que l'hyperplan $H=D\cap (-D)$ ne contient aucun élément de $R$. On définit alors $R^+=D\cap R$ comme l'ensemble des racines positives de $R$ associé à $D$.

L'ensemble $B(R^+)=\{ \alpha\in R^+ \vert \alpha\text{ ne peut pas s'écrire }\alpha=\alpha_1+\alpha_2\text{ pour }\alpha_1,\alpha_2\in R^+\}$ vérifie que tout élément de $R$ est combinaison linéaire à coefficients entiers de même signe de $B(R^+)$. L'ensemble $B(R^+)$ est appelé le système de racines simples de $R$ associé à $R^+$
\end{propdef}

\begin{proof}
Le seul point à vérifier est que tout élément de $R$ est combinaison linéaire à coefficients entiers de même signe de $B(R^+)$, ce qui provient de \cite[ch. VI, théorème 3]{Bo}.
\end{proof}

Les systèmes de racines de réseaux pairs sont toujours isomorphes à des unions disjointes des systèmes de racines des réseaux $\mathrm{A}_n,\mathrm{D}_n,\mathrm{E}_8,\mathrm{E}_7,\mathrm{E}_6$ que l'on décrit ci-dessous.
\begin{enumerate}
\item[$\mathrm{A}_n$] : On pose $\mathrm{A}_n = \{ (x_i)\in \Z^{n+1} \vert \sum_i x_i=0\}$. On a $\mathbf{A}_n=R(\mathrm{A}_n)=\{ \pm (e_i-e_j) \vert i\neq j\}$.
\item[$\mathrm{D}_n$] : On pose $\mathrm{D}_n = \{ (x_i)\in \Z^n \vert \sum_i x_i\equiv 0\mathrm{\ mod\ }2\}$. On a $\mathbf{D}_n=R(\mathrm{D}_n)=\{ \pm e_i \pm e_j \vert i\neq j\}$. 
\item[$\mathrm{E}_8$] : On pose $\mathrm{E}_8 = \mathrm{D}_8+ \Z \cdot e$, avec $e=\frac{1}{2}(1,\dots ,1)$. On a $\mathbf{E}_8=R(\mathrm{E}_8)=R(\mathrm{D}_8) \cup \left\{ (x_i)=\frac{1}{2} (\pm 1, \dots, \pm 1) \vert \prod _i x_i >0 \right\}$. 
\item[$\mathrm{E}_7$] : On pose $\mathrm{E}_7 = e^\perp \cap \mathrm{E}_8 =\left\{ (x_i)\in \mathrm{E}_8 \vert \sum_i x_i =0 \right\}$. On a $\mathbf{E}_7=R(\mathrm{E}_7)=e^\perp \cap R(\mathrm{E}_8) =  R(A_7)\cup \left\{ (x_i)=\frac{1}{2} (\pm 1, \dots, \pm 1) \vert \sum _i x_i =0 \right\} $. 
\item[$\mathrm{E}_6$] : On pose $\mathrm{E}_6=(e_7+e_8)^\perp \cap \mathrm{E}_7$. On a $\mathbf{E}_6=R(\mathrm{E}_6)=(e_7+e_8)^\perp \cap R(\mathrm{E}_7)$.
\end{enumerate} 

Suivant ces notations, on a les isomorphismes : $\mathbf{D}_1\simeq \mathbf{A}_1$, $\mathbf{D}_2\simeq (\mathbf{A}_1)^2$ et $\mathbf{D}_3\simeq \mathbf{A}_3$, donc on n'utilisera la notation $\mathbf{D}_n$ que pour $n\geq 4$.

De plus, les systèmes de racines $\mathbf{A}_n$ ($n\geq 1$), $\mathbf{D}_n$ ($n\geq 4$), $\mathbf{E}_8$, $\mathbf{E}_7$ et $\mathbf{E}_6$ sont deux-à-deux non isomorphes, et ce sont (à isomorphisme près) les seuls systèmes de racines irréductibles de type ADE (au sens de \cite[ch. VI, \S 1]{Bo}).

\begin{defi}[Les ensembles $\mathcal{L}_n$ et $X_n$] Soit $n\equiv 0,\pm 1\ \mathrm{mod}\ 8$. On définit $\mathcal{L}_n$ comme l'ensemble des réseaux pairs $L\subset V$ tels que $\mathrm{det}(L)=1$ si $n$ est pair et $\mathrm{det}(L)=2$ sinon.

\`A $n$ fixé, le groupe orthogonal euclidien $\mathrm{O}_n(\R)$ agit naturellement sur l'ensemble $\mathcal{L}_n$, et on note $X_n$ l'ensemble des classes d'isomorphisme des éléments de $\mathcal{L}_n$, qui est un ensemble fini.
\end{defi}

On rappelle que $\mathcal{L}_n$ est non vide pour $n\equiv 0,\pm 1\mathrm{\ mod\ }8$. Par exemple, suivant les notations précédentes, $\mathcal{L}_n$ contient :

- le réseau $\mathrm{E}_8^{(n-7)/8}\oplus \mathrm{E}_7$ si $n\equiv -1\mathrm{\ mod\ }8$ ;

- le réseau $\mathrm{E}_8^{n/8}$ si $n\equiv 0\mathrm{\ mod\ }8$ ;

- le réseau $\mathrm{E}_8^{(n-1)/8}\oplus \mathrm{A}_1$ si $n\equiv 1\mathrm{\ mod\ }8$.

\begin{lem} Soient $L\subset V$ un réseau pair, et $R=R(L)$ son système de racines. Si l'on possède $R^+\subset R$ un système de racines positives, et que l'on note $\rho=\frac{1}{2}\sum_{\alpha\in R^+} \alpha$, alors le système de racines simples associé à $R^+$ est donné par : $\{ \alpha\in R^+ \, \vert\, \alpha\cdot \rho =1\}.$
\end{lem}

\begin{proof}
Découle directement de \cite[ch.VI, proposition 29]{Bo}.
\end{proof}

\begin{lem}\label{Rr}Soient $R$ un système de racine irréductible de type ADE, et $r\in R$. Alors $R\cap r^\perp$ est un système de racine (éventuellement vide) donné par la table suivante :
\begin{center} \begin{tabular}{|c|c|}
\hline
$R$ & $R\cap r^\perp$ \\
\hline
$\mathbf{A}_n$ ($n\leq 2$) & $\emptyset$ \\
$\mathbf{A}_n$ ($n\geq 3$) & $\mathbf{A}_{n-2}$ \\
$\mathbf{D}_n$ ($n\geq 4$) & $\mathbf{D}_{n-2}\coprod \mathbf{A}_1$ \\
$\mathbf{E}_8$ & $\mathbf{E}_{7}$ \\
$\mathbf{E}_7$ & $\mathbf{D}_{6}$ \\
$\mathbf{E}_6$ & $\mathbf{A}_{5}$ \\
\hline
\end{tabular}
\end{center}
\end{lem}

\begin{proof}
Le groupe de Weyl de $R$ agit transitivement sur l'ensemble des éléments de $R$. La classe d'isomorphisme de $R\cap r^\perp$ ne dépend donc uniquement de la classe d'isomorphisme de $R$, et non de la racine $r$ choisie. Il suffit donc de vérifier le tableau pour les réseaux $\mathbf{A}_n,\mathbf{D}_n,\mathbf{E}_8,\mathbf{E}_7$ et $\mathbf{E}_6$ décrits précédemment, en prenant une racine quelconque $r\in R$, ce que l'on fait facilement à la main.
\end{proof}

\begin{lem} Soit $n\geq 1$ et $L\subset V$ un réseau pair. Si l'on se donne $r\in R(L)$, alors $L'=L\cap r^\perp$ est un sous-$\Z$-module de rang $n-1$ de $L$, et c'est un réseau pair de l'espace $V\cap r^\perp$. De plus, le système de racines de $L'$ est donné par : $R(L') = R(L) \cap r^\perp$. En particulier, la classe d'isomorphisme de $R(L')$ ne dépend que de $R(L)$ et de la composante irréductible de $R(L)$ contenant $r$, et elle se déduit du lemme \ref{Rr}.
\end{lem}

\begin{proof}
On considère la décomposition en composantes irréductibles du système de racines $R(L)$ : 
$$R(L)\simeq \coprod_i R_i$$
où les $R_i$ sont des systèmes de racines irréductibles de type ADE (dont certains peuvent être égaux).

Si l'on se donne $r\in R_j$, alors par définition on a : $\coprod_{i\neq j}R_j \subset r^\perp$. Ainsi, le système de racines de $L'=L\cap r^\perp$ vérifie : $R(L')\simeq \left( \coprod_{i\neq j} R_j \right) \coprod \left( R_j\cap r^\perp \right)$, et la classe d'isomorphisme de $R_j\cap r^\perp$ est donnée par le lemme précédent.
\end{proof}

\subsection{Les opérateurs de Hecke et les $A$-voisins.}\label{2.2}
Commençons par rappeler la définition des $A$-voisins :
\begin{propdef}[Les $A$-voisins] Soient $A$ un groupe abélien fini, et $L_1,L_2$ deux éléments de $\mathcal{L}_n$. Les conditions suivantes sont équivalentes :

\begin{enumerate}
\item[$(i)$] Le quotient $L_1/(L_1\cap L_2)$ est isomorphe à $A$.
\item[$(ii)$] Le quotient $L_2/(L_1\cap L_2)$ est isomorphe à $A$.
\end{enumerate}

Si ces conditions sont vérifiées, on dit que $L_1$ et $L_2$ sont des $A$-voisins, ou que $L_2$ est un $A$-voisin de $L_1$.
\end{propdef}

\begin{proof}
Voir \cite[ch.III, \S1]{CL} et \cite[Annexe B,\S 3]{CL} selon la parité de $n$.
\end{proof}

Dans le cas particulier où $A$ est de la forme $\Z/d\Z$, on parlera de $d$-voisin : c'est ce cas qui nous intéressera plus particulièrement. Il est alors facile de construire l'ensemble des $d$-voisins d'un réseau $L$ donné :

\begin{prop} \label{dvoisin} Soient $L\in \mathcal{L}_n$ et $d\in \N^*$. On note $C_L(\Z/d\Z)$ l'ensemble des droites isotropes de $L/dL$ (où on entend par droite un $\Z/d$-module libre de dimension $1$). Alors les d-voisins de $L$ sont en bijection naturelle avec les points de la quadrique $C_L (\Z/d\Z)$ comme suit.

Donnons-nous $x$ une droite isotrope de $L/dL$ et $v\in L$ dont l'image dans $L/dL$ engendre $x$ vérifiant $v\cdot v\equiv 0\mathrm{\ mod\ }2d^2$, et notons $M$ l'image réciproque de $x^\perp$ par l'homomorphisme $L\rightarrow L/dL$. Alors le réseau $M+\Z \frac{v}{d}$ est un $d$-voisin de $L$ ne dépendant que de $x$ : on le note $L'(x)$.

L'application $x\mapsto L'(x)$ est une bijection entre $C_L(\Z/d\Z)$ et l'ensemble $\mathrm{Vois}_d(L)$ des $d$-voisins de $L$. Le réseau $L'(x)$ sera appelé le $d$-voisin de $L$ associé à $x$
\end{prop}

\begin{proof}
Voir \cite[ch. III, \S 1]{CL} et \cite[Annexe B, \S 3]{CL} selon la parité de n.
\end{proof}

\begin{propdef} Soient $n\equiv 0,\pm 1\mathrm{\ mod\ }8$, et $L\in \mathcal{L}_n$. Alors le cardinal de la quadrique $C_L (\Z/d\Z)$ ne dépend que de $n$ et de $d$, et on le notera $c_n(d)$.

En particulier, pour $p$ premier, on a le résultat suivant :
$$c_n(p)= \left\{ \begin{array}{ll}
\sum_{i=1}^{n-2}p^i +p^{\frac{n}{2}-1} & \text{ si $n$ est pair ;}\\
\sum_{i=1}^{n-2}p^i & \text{ si $n$ est impair.}\\
\end{array} \right.$$

Le calcul de $c_n(d)$ pour $d\in \Z$ quelconque se déduit des constatations suivantes :
\begin{enumerate}
\item[$(i)$] $c_n(d_1\, d_2)=c_n(d_1)\, c_n(d_2)$ si $d_1$ et $d_2$ sont premiers entre eux ;
\item[$(ii)$] $c_n(p^k) = p^{(k-1)(n-2)}\, c_n(p)$ pour $p$ premier et $k\in \N$.
\end{enumerate}
\end{propdef}
\begin{proof}
Le calcul de $c_n(p)$ pour $p$ premier se déduit de \cite[ch. III]{CL} et \cite[Annexe B]{CL} selon la parité de $n$.

Le point $(i)$ se déduit de la bijection entre $L/d_1 d_2 L$ et $L/d_1 L \times L/d_2 L$ et du lemme des restes chinois.

Le point $(ii)$ se déduit de la surjection $C_L(\Z/p^k\Z)\twoheadrightarrow C_L(\Z/p^{k-1}\Z)$. Si l'on se donne une droite isotrope $x\in C_L(\Z/p^{k-1}\Z)$ engendré par un vecteur $v\in L\setminus pL$, la fibre au dessus de $x$ est un espace affine dirigé par $v^\perp /\F_p v$, où $v^\perp = \{ w\in L/pL \, \vert\, (v\cdot w) \equiv 0\mathrm{\ mod\ }p\}$. En particulier, ces fibres sont toutes de cardinal $p^{n-2}$, et une récurrence sur $k$ donne le résultat cherché. 
\end{proof}

Notons $\Z [X_n]$ le $\Z$-module libre engendré par l'ensemble $X_n$. On définit sur $\Z [X_n]$ les endomorphismes suivants :
\begin{defi}[Les opérateurs de Hecke] Si $L\in \mathcal{L}_n$, on note $\overline{L}$ sa classe dans $X_n$. On note de plus $\mathrm{Vois}_A(L)$ l'ensemble des $A$-voisins de $L$, et pour tout élément $L'\in \mathrm{Vois}_A(L)$, on note $\overline{L'}$ sa classe dans $X_n$.

L'opérateur de Hecke $\mathrm{T_A}$ est l'endomorphisme de $\Z[X_n]$ défini par :
$$\mathrm{T}_A (\overline{L}) = \sum_{L'\in \mathrm{Vois}_A(L)} \overline{L'},$$
pour tout réseau $L\in \mathcal{L}_n$.
\end{defi}

Posons $N=\vert X_n \vert$, et donnons-nous $L_1,\dots,L_N\in \mathcal{L}_n$ d'image deux-à-deux distinctes $\overline{L}_1,\dots,\overline{L}_N$ dans $X_n$. D'après la définition précédente, si l'on note $T_A=t_{i,j}\in \mathrm{M}_N(\Z)$ la matrice de $\mathrm{T}_A$ dans la base $\overline{L}_1,\dots,\overline{L}_N$, alors le coefficient $t_{i,j}$ est le nombre de $A$-voisins de $L_j$ isomorphes à $L_i$.

Pour simplifier, on notera $\mathrm{T}_d=\mathrm{T}_{\Z/d\Z}$. En particulier, si l'on note $T_d=(t_{i,j})$ la matrice de $\mathrm{T}_d$ dans la base $\overline{L}_1,\dots,\overline{L}_N$ (suivant les notations précédentes), alors on a :
$$(\forall\, j\in \{1,\dots,N\})\ \sum_i t_{i,j}=c_n(d).$$ 

\begin{defi}[Le graphe de Kneser] Soient $p$ un nombre premier, et $n\equiv 0,\pm 1\mathrm{\ mod\ }8$. Le graphe des $p$-voisins $\mathrm{K}_n(p)$ est le graphe défini de la manière suivante :

- l'ensemble des sommets est l'ensemble $X_n$ ;

- l'ensemble des arêtes est l'ensemble des $\{ \overline{L_1},\overline{L_2}\}$, pour $L_1,L_2\in \mathcal{L}_n$ des $p$-voisins (où pour $i=1,2$ on désigne par $\overline{L_i}$ la classe de $L_i$ dans $X_n$).
\end{defi}

\begin{prop} \label{kneserconn} Le graphe de $\mathrm{K}_n(p)$ est connexe pour tout $n$ et pour tout $p$.
\end{prop}
\begin{proof}
Le cas où $n$ est pair est démontré dans \cite[ch. III, Théorème 1.12]{CL}. Le cas où $n$ est impair se traite exactement de la même manière.
\end{proof}

Les opérateurs de Hecke $\mathrm{T}_A$ participent à la notion plus générale d'anneau de Hecke d'un schéma en groupes affine sur $\Z$ de type fini. Si l'on se donne $G$ un tel schéma en groupe, on peut lui associer son anneau de Hecke défini comme suit :
\begin{defi}[L'anneau des opérateurs de Hecke] Soit $\Gamma$ un groupe, et soit $X$ un $\Gamma$-ensemble transitif. On définit l'anneau des opérateurs de Hecke de $X$ comme le sous-anneau $\mathrm{H}(X)\subset \mathrm{End}_\Z (\Z[X])$ des endomorphismes commutant à l'action de $\Gamma$.
\end{defi}
\begin{defi}[L'anneau de Hecke d'un schéma en groupe] Soit $G$ schéma en groupes affine sur $\Z$ de type fini. Si l'on note $P$ l'ensemble des nombres premiers, on note $\widehat{\Z}=\prod_{p\in P}\Z_p$, et $\mathbb{A}_f =\Q \otimes \widehat{\Z}$ l'anneau des adèles finis de $\Q$. On définit alors le $G(\mathbb{A}_f)$-ensemble : $\mathcal{R}(G)=G(\mathbb{A}_f)/G(\widehat{\Z})$. L'anneau de Hecke de $G$ est alors défini comme :
$$\mathrm{H}(G) = \mathrm{H}(\mathcal{R}(G))$$
où $G(\mathbb{A}_f)$ joue le rôle de $\Gamma$ dans la définition précédente.
\end{defi}

En particulier, on s'intéressera aux cas où $G=\mathrm{O}_n$ ou $G=\mathrm{SO}_n$, définis comme suit :

\begin{defi} Si on se donne $L_0$ un élément de $\mathcal{L}_n$, on définit $\mathrm{O}_n$ le schéma en groupes affine sur $\Z$ associé à la forme quadratique $L_0 \rightarrow\Z$, $x\mapsto q(x)$. Il s'agit de l'objet noté $\mathrm{O}_{L_0}$ dans \cite[ch. II, \S1]{CL}. On définit de même $\mathrm{SO}_n \subset \mathrm{O}_n$ (introduit aussi dans dans \cite[ch. II, \S1]{CL}). Les schémas $\mathrm{O}_n$ et $\mathrm{SO}_n$ ainsi définis sont des schéma en groupes affine sur $\Z$ de type fini (ce dernier étant même réductif).
\end{defi}

Les anneaux de Hecke $\mathrm{H}(\mathrm{O}_n)\subset \mathrm{H}(\mathrm{SO}_n)$ sont alors bien définis. Considérons $G=\mathrm{O}_n$, et donnons-nous $L_0\in \mathcal{L}_n$. Alors l'ensemble $\mathcal{R}(G)$ s'identifie à l'ensemble des réseaux de $L_0\otimes \Q$ qui sont dans $\mathcal{L}_n$. Cette identification permet de voir les opérateurs de Hecke $\mathrm{T}_A$ introduits précédemment comme des éléments de l'algèbre de Hecke $\mathrm{H}(G)$. On a alors la proposition suivante :

\begin{prop} \label{H(G)} Soient $n\equiv 0,\pm 1\mathrm{\ mod\ }8$, et $G=\mathrm{O}_n$. Alors :

\begin{enumerate}
\item[$(i)$] Les opérateurs de Hecke $\mathrm{T}_A$ associés aux $A$-voisins forment une $\Z$-base de l'anneau de Hecke $\mathrm{H}(G)$.
\item[$(ii)$] L'anneau $\mathrm{H}(G)$ est commutatif.
\end{enumerate}
\end{prop}
\begin{proof}
Voir \cite[ch. IV, \S 2.6]{CL}.
\end{proof}

\subsection{La paramétrisation de Langlands-Satake.} \label{2.3}
Dans toute la suite, $G$ désignera un schéma en groupe affine sur $\Z$ de type fini et semi-simple. Notons $\widehat{G}$ son dual au sens de Langlands. C'est un $\C$-groupe réductif dont la donnée radicielle est duale à celle de $G(\C)$, suivant Borel \cite{Bor77} et Springer \cite{Spr79} par exemple. Notons de plus $\widehat{\mathfrak{g}}$ l'algèbre de Lie complexe de $\widehat{G}$, et $\widehat{G}(\C)_{\mathrm{ss}}$ et $\widehat{\mathfrak{g}}(\C)_{\mathrm{ss}}$ les classes de $\widehat{G}(\C)$-conjugaison d'éléments semi-simples respectivement de $\widehat{G}(\C)$ et $\widehat{\mathfrak{g}}(\C)$.

\`A la manière de \cite[ch. IV, \S 3.2]{CL}, on note $\Pi (G)$ l'ensemble des classes d'isomorphisme de représentations unitaires irréductibles $\pi$ de $G(\mathbb{A})$ telles que $\pi_p$ est non ramifiée pour tout $p$ premier. Nous noterons également $\Pi_{\mathrm{cusp}}(G)$ et $\Pi_{\mathrm{disc}}(G)$ respectivement l'ensemble des représentations automorphes cuspidales et l'ensemble des représentations automorphes discrètes de $G$, suivant les notations de \cite{CL} et \cite{CR}.

On désigne par $P$ l'ensemble des nombres premiers, et on définit $\mathcal{X}(\widehat{G})$ l'ensemble des familles $(c_v)_{v\in P\cup \{\infty \} }$, où $c_\infty \in \widehat{\mathfrak{g}}_{\mathrm{ss}}$ et $c_p \in \widehat{G}(\C)_{\mathrm{ss}}$ pour tout $p\in P$. Suivant Langlands dans \cite{Lan}, on dispose d'une application canonique $c :\Pi (G)\rightarrow \mathcal{X}(\widehat{G}), \pi \mapsto (c_v(\pi))$. L'élément $c_\infty(\pi)$ est appelé le caractère infinitésimal de $\pi$. Lorsque $G=\mathrm{PGL}_n$, auquel cas on a $\widehat{G}=\mathrm{SL}_n(\C)$, les valeurs propres du caractère infinitésimal de $\pi$ sont bien définies et sont appelées les poids de $\pi$.

Si l'on possède $r:\widehat{G}\rightarrow \mathrm{SL}_n$ une $\C$-représentation, celle-ci induit une application $\mathcal{X}(\widehat{G}) \rightarrow \mathcal{X}(\mathrm{SL}_n),(c_v)\mapsto (r(c_v))$. Pour $\pi\in \Pi(G)$, on note $\psi (\pi,r)=r(c(\pi))$ : c'est un élément de $\mathcal{X}(\mathrm{SL}_n)$ qu'on appelle le paramètre de Langlands du couple $(\pi,r)$.

En pratique, on considérera le cas où $G$ est le groupe $\mathrm{SO}_n$, et $\widehat{G}$ est donc un groupe de la forme $\mathrm{SO}_m$ ou $\mathrm{Sp}_{2m}$. On utilisera alors le paramètre de Langlands du couple $(\pi,{\mathrm{St}})$, où ${\mathrm{St}}$ désigne la représentation standard de $\widehat{G}$, et on parlera alors du paramètre standard de $\pi$.\\

Soit $k\in \N^*$. Pour tout $i\in \{1,\dots,k\}$, donnons-nous $n_i,d_i \in \N^*$ et $\pi_i \in \Pi_{\mathrm{cusp}}(\mathrm{PGL}_{n_i})$ tels que $\sum_i n_i d_i =n$. On dispose alors d'un élément de $\mathcal{X}(\mathrm{SL}_n)$, noté $\oplus_i \pi_i [d_i]$ dans \cite{CL} ou \cite{CR}. Par définition, les paramètres de Langlands de $\oplus_i \pi_i [d_i]$ satisfont les égalités suivantes :

$$ (\forall\, v\in P\cup \{\infty \}) \ c_v (\oplus_i \pi_i [d_i]) = \bigoplus_i c_v (\pi_i) \otimes \mathrm{Sym}^{d_i-1} (e_v)$$

où $e_\infty = \left( \begin{matrix}
-\frac{1}{2} & 0 \\
0 & \frac{1}{2}\\
\end{matrix} \right)$ et $e_p = \left( \begin{matrix}
p^{-\frac{1}{2}} & 0 \\
0 & p^{\frac{1}{2}} \\
\end{matrix} \right) $.

\`A la manière de \cite[Ch. VI, \S 4]{CL} on note $\mathcal{X}_{\mathrm{AL}}(\mathrm{SL}_n)$ l'ensemble des éléments de la forme $\oplus_i \pi_i [d_i]$, suivant ces notations. Avec ce formalisme, la conjecture d'Arthur-Langlands \cite{Lan} \cite{Art89} se formule facilement (voir \cite[ch. VI, \S4, Conjecture 4.6]{CL}). Lorsque $G=\mathrm{SO}_n$ et que l'on considère la représentation standard de $\widehat{G}$, cette conjecture a été vérifiée par Taïbi \cite{Tai16}, dont les résultats reposent sur les travaux d'Arthur \cite{Art89}, ainsi que sur ceux de Kaletha \cite{Kal} et Arancibia-Moeglin-Renard \cite{AMR}. Ainsi, le paramètre standard d'une représentation $\pi\in \Pi_{\mathrm{disc}}(\mathrm{SO}_n)$ est un élément de $\mathcal{X}_{\mathrm{AL}}(\mathrm{SL}_n)$.

Si l'on se donne $\pi\in \Pi_{\mathrm{disc}}(\mathrm{SO}_n)$, alors on possède une égalité de la forme $\psi(\pi,\mathrm{St})=\bigoplus_i \pi_i [d_i]$. On définit les poids de $\pi$ comme les valeurs propres du caractère infinitésimal de $\psi(\pi,\mathrm{St})$. Une étude de ce caractère infinitésimal nous dit que les $\pi_i$ précédents sont des éléments de $\Pi_{\mathrm{alg}}^\bot (\mathrm{PGL}_{n_i})$ (suivant les notations de \cite{CL} ou \cite{CR} déjà exposées en introduction).

Enfin, nous adopterons les notations suivantes. Considérons $n\in\{1,2,3\}$, et $w_1>\dots >w_n>0$ des entiers de même parité. Si l'on désigne $\Pi$ l'ensemble des éléments de $\Pi_{\mathrm{alg}}^\bot (\mathrm{PGL}_{2n})$ de poids l'ensemble $\{\pm \frac{w_1}{2},\dots,\pm \frac{w_n}{2}\}$. L'ensemble $\Pi$ est fini. Notons $m$ son cardinal. On note $\Delta_{w_1,\dots,w_n}$ son unique élément lorsque $m=1$, et $\Delta_{w_1,\dots,w_n}^m$ n'importe lequel de ses éléments sinon. Lorsque $m=1$, on définit la fonction $D_{w_1,\dots,w_n}(p) = p^{\frac{w_1}{2}} \mathrm{Trace}\left( c_p(\Delta_{w_1,\dots,w_n}) \vert V_{\mathrm{St}} \right)$. Si $m>1$, on définit l'ensemble de fonctions $D^m_{w_1,\dots,w_n}(p)=\{ p^{\frac{w_1}{2}} \mathrm{Trace}\left( c_p(\Delta^m_{w_1,\dots,w_n}) \vert V_{\mathrm{St}} \right)\}$.\\

Lorsque $n=1$, les fonctions $D^m_{w_1}$ se comprennent bien à l'aide des formes modulaires pour $\mathrm{SL}_2(\Z)$. Par exemple, on peut considérer les cas où $w_1\in \{11,15,17,19,21\}$ (auxquels cas $m=1$ suivant les notations précédentes). Notons $\tau_{w_1+1}(n)$ le $n$-ème terme du $q$-développement de l'unique forme modulaire normalisée de poids $w_1+1$ pour $\mathrm{SL}_2(\Z)$. Alors on a pour tout $p$ premier l'égalité : $\tau_{w_1+1}(p)=D_{w_1}(p)$.

On considérera aussi le cas $w_1=23$ (auquel cas $m=2$). Notons $\mathrm{E}_k$ la série d'Eisenstein normalisée de poids $k$, et $\Delta=\frac{1}{1728}(\mathrm{E}_4^3-\mathrm{E}_6^2)$ la fonction de Jacobi. Alors les fonctions $\Delta\, \mathrm{E}_4^3$ et $\Delta\, \mathrm{E}_6^2$ forment une base de l'espace des formes modulaires paraboliques de poids $24$. On définit les fonctions $\tau_{24}^\pm (n)$ comme étant le $n$-ème terme du $q$-développement de la forme modulaire parabolique normalisée de poids $24$ suivante :

$$\frac{131 \pm \sqrt{144169}}{144} \Delta\, \mathrm{E}_4^3 + \frac{13 \mp \sqrt{144169}}{144} \Delta\, \mathrm{E}_6^2 .$$

Alors on a pour tout $p$ premier l'égalité : $D^2_{23}(p)=\{ \tau_{24}^+ (p), \tau_{24}^- (p) \}$.\\

Lorsque $n=2$, les fonctions $D^m_{w_1,w_2}$ se comprennent grâce aux formes modulaires de Siegel de genre $2$. Dans la suite, on considérera les cas où $(w_1,w_2) \in \{ (19,7) , (21,5) , (21,9) , (21,13), \newline (23,7) , (23,9) , (23,13) \}$, auxquels cas $m=1$. En reprenant les notations de \cite[Introduction]{CL}, on a pour tout nombre premier $p$ et tout couple $(w_1,w_2)$ dans l'ensemble précédent l'égalité : $\tau_{w_2-1,\frac{w_1-w_2+4}{2}}(p) = D_{w_1,w_2}(p)$.\\

On utilisera notamment la proposition suivante :
\begin{prop}[Les inégalités de Ramanujan] \label{rama} Soient $n\in\{1,2,3\}$, $w_1>\dots>w_n>0$ des entiers de même parité, $\Pi$ l'ensemble des éléments de $\Pi_{\mathrm{alg}}^\bot (\mathrm{PGL}_{2n})$ de poids $\{\pm\frac{w_1}{2},\dots,\pm \frac{w_n}{2}\}$, avec $m=\vert \Pi\vert$, et $p$ un nombre premier.

Si $m=1$, alors : $\vert D_{w_1,\dots,w_n}(p) \vert \leq 2\,n\,p^{\frac{w_1}{2}}$.

Si $m>1$, alors : $(\forall D\in D^m_{w_1,\dots,w_n}(p)) \vert D\vert \leq 2\,n\,p^{\frac{w_1}{2}}$.
\end{prop}

\begin{proof} D'après \cite{CHL11}, \cite{Shin11}, \cite{Ca12}, \cite{CH13} et \cite{Clo13}, si l'on se donne $\Delta\in \Pi$, alors $\Delta$ satisfait la conjecture de Ramanujan. Ainsi, les valeurs propres de $c_p(\Delta)$ sont toutes de module $1$, et donc $\vert \mathrm{Trace}(c_p(\Delta) \vert \mathrm{St}) \vert\leq 2\,n$, puis $\vert p^{\frac{w_1}{2}} \mathrm{Trace}(c_p(\Delta) \vert \mathrm{St}) \vert \leq 2\, n\, p^{\frac{w_1}{2}}$, qui est l'inégalité cherchée.
\end{proof}
\section{Calcul de la matrice de $\mathrm{T}_2$ sur $\Z[X_n]$ pour $n=23$ ou $25$.} \label{3}
\subsection{L'étude des systèmes de racines d'éléments de $X_n$.}\label{3.1}

La proposition suivante était probablement déjà connue de Borcherds :
\begin{prop}\label{RXn} Soient $n=23$ ou $25$, et $L_1,L_2\in \mathcal{L}_n$. Pour $i=1,2$, on note $R_i=R(L_i)$ le système de racines de $L_i$. Alors on a l'équivalence :
$$L_1\simeq L_2 \Leftrightarrow R_1\simeq R_2$$
\end{prop}
\begin{proof}
L'implication $L_1\simeq L_2\Rightarrow R_1\simeq R_2$ est évidente. Il suffit donc de montrer que $L_1\not\simeq L_2\Rightarrow R_1\not\simeq R_2$, c'est-à-dire que deux éléments non isomorphes de $\mathcal{L}_n$ ont des systèmes de racines non-isomorphes, ce qui se fait par inspection selon la valeur de $n$.

Si $n=23$ : d'après \cite[Annexe B]{CL}, on sait construire $X_{23}$ à l'aide de $X_{24}$ comme suit. Si on se donne $P\in \mathcal{L}_{24}$ qui n'est pas isomorphe au réseau de Leech, et $r\in R(P)$, alors le réseau $L = P\cap r^\perp$ est un élément de $\mathcal{L}_{23}$, et les classes d'isomorphisme de réseaux ainsi obtenus décrivent tout $X_{23}$. Pour $i=1,2$, donnons-nous $P_i\in \mathcal{L}_{24}$ et $r_i\in R(P_i)$, et notons $R_i$ la composante irréductible de $R(P_i)$ à laquelle appartient $r_i$ et notons $L_i=P_i\cap r_i^\perp \in \mathcal{L}_{23}$ : alors les réseaux $L_1$ et $L_2$ sont isomorphes si, et seulement si, les couples $(P_1,R_1)$ et $(P_2,R_2)$ sont isomorphes (d'après \cite[Annexe B, prop. 2.6]{CL}).

Il suffit donc de regarder pour chaque élément de $X_{23}$ ainsi généré si les classes d'isomorphisme des systèmes de racines sont deux-à-deux distincts. On détermine ces systèmes de racines grâce aux lemmes précédents, que l'on donne dans le table \ref{RX23}, où les notations sont les suivantes : $P$ est un élément de $\mathcal{L}_{24}$ qui n'est pas isomorphe au réseau de Leech (dont la classe d'isomorphisme est entièrement déterminée par celle de son système de racines $R(P)$), $r$ est un élément de $R(P)$ appartement à la composante irréductible $R$, $L=P\cap r^\perp$ est l'élément de $\mathcal{L}_{23}$ qui nous intéresse, et $R(L)$ est son système de racines.

L'inspection de la table \ref{RX23} montre bien le résultat cherché, à savoir que deux éléments non isomorphes de $\mathcal{L}_{23}$ ont des systèmes de racines non isomorphes.\\

Si $n=25$ : on renvoie à \cite{Bor} pour la liste des classes d'isomorphisme des éléments de $\mathcal{L}_{25}$ et à leur systèmes de racines. Une inspection rapide montre que l'on a bien l'équivalence cherchée.
\end{proof}

\subsection{Détermination de la classe d'isomorphisme d'un $2$-voisin d'un élément de $\mathcal{L}_n$.}\label{3.2}
D'après la proposition \ref{RXn}, la table \ref{RX23} (pour $n=23$) et les résultats de Borcherds \cite{Bor} (pour $n=25$) nous permettent de déterminer la classe dans $X_n$ d'un élément de $\mathcal{L}_n$ grâce à son système de racines. Cependant, déterminer la classe d'isomorphisme d'un système de racine peut s'avérer long sur le plan algorithmique. En pratique, on utilisera la proposition suivante :

\begin{prop}\label{Phi} On fixe $n=23$ ou $25$. Soient $L\in \mathcal{L}_n$ et $R(L)$ son système de racines. Soient $\{\alpha_1,\dots,\alpha_k\}$ un ensemble de racines simples de $R(L)$, et $A=(\alpha_i \cdot \alpha_j) = (a_{i,j})\in \mathrm{M}_k(\Z)$ la matrice de Gram associée. Pour $l\in \Z$, on définit les entiers : $n_l = \left\vert \left\{ i\in \{1,\dots ,l\} \,\vert\, \sum_j a_{i,j}=l \right\} \right\vert$, qui ne dépendent que de la classe $\overline{L}$ de $L$ dans $X_n$. De même, les quantités $\vert R(L)\vert$ et $\mathrm{det}(A)$ ne dépendent que de $\overline{L}$.

On définit alors les applications $\Phi_n:\mathcal{L}_n\rightarrow \N^6$ et $\overline{\Phi}_n:X_n\rightarrow \N^6$ par :
$$\Phi_n(L) = \overline{\Phi}_n(\overline{L})=(\vert R(L) \vert,n_2,n_1,n_0,n_{-1},\mathrm{det}(A)).$$

L'application $\overline{\Phi}_n$ ainsi définie est injective sur $X_n$.
\end{prop}
\begin{proof}
Par définition du système de racines $R(L)$ d'un réseau $L$, il est immédiat que éléments isomorphes de $\mathcal{L}_n$ ont même image par $\Phi_n$, ce qui prouve que $\overline{\Phi}_n$ est bien définie.

Pour l'injectivité de $\overline{\Phi}_n$, il suffit de vérifier que deux éléments distincts dans $X_n$ ont des images distinctes par $\overline{\Phi}_n$, ce qui se fait facilement à la main (comme on connaît déjà les classes d'isomorphisme des systèmes de racines de tous les éléments de $X_n$). Les tables \ref{rX23} et \ref{rX25} donnent pour tout élément $\overline{L_i}$ de $X_n$ la classe $R_i$ du système de racines de $L_i$, ainsi que l'image $\phi_i$ de $\overline{L_i}$ par $\overline{\Phi}_n$.
\end{proof}

Notons que les quantités $n_2,n_1,n_0,n_{-1}$ se comprennent très bien grâce au diagramme de Dynkin de $R(L)$. Il s'agit respectivement du nombre de sommet possédant aucun, un, deux ou trois sommets qui lui sont connexes. L'intérêt de la définition des $n_l$ donnée à la proposition \ref{Phi} est qu'elle indique la manière dont on les calcule dans nos algorithmes.\\

D'après les propositions \ref{RXn} et \ref{Phi}, pour déterminer la classe d'isomorphisme d'un élément de $\mathcal{L}_n$, il suffit de déterminer son image par $\Phi_n$. On souhaite donc déterminer l'image par $\Phi_n$ d'un $2$-voisin d'un réseau $L$. On utilise pour cela le lemme suivant :
\begin{lem} \label{Lx} On fixe $n=23$ ou $25$. Donnons-nous $L\in \mathcal{L}_n$, et considérons $x$ la droite isotrope de $L/2L$ engendrée par $v\in L$. On note $L_x$ le $2$-voisin de $L$ associé à la droite $x$ suivant \cite[ch.III, proposition 1.4]{CL}. On suppose enfin que l'on possède une base $\{a_1,\dots,a_n\}$ de $L$.
\\ \indent Si l'on se donne $j\in \{1,\dots,n\}$ tel que $(a_j\cdot v)\equiv 1\mathrm{\ mod\ }2$, alors la famille :
$$\{(a_j\cdot v)\, a_i - (a_i\cdot v)\, a_j \vert i\in \{1,\dots ,n\} \cup \{2\, a_i \vert i\in \{1,\dots,n\} \} \cup \left\{ \frac{1}{2}\left( v-\frac{(v\cdot v)}{2}a_j\right) \right\}$$
engendre $\Z$-linéairement $L_x$.
\end{lem}
\begin{proof}
On reprend la construction faite dans \cite{CL}. Le réseau $L_x$ est engendré par le réseau $M$ (où $M$ est l'image réciproque de $x^\perp$ par la projection $L\rightarrow L/2L$) et par le vecteur $\frac{1}{2}v'$ (où $v'$ est un élément de $L$ vérifiant $(v'\cdot v')\equiv 0\mathrm{\ mod\ }8$, dont l'image dans $L/2L$ engendre $x$).

Le premier point à vérifier est l'existence de l'entier $j$ du lemme. Celle-ci provient de la non-dégénérescence du produit scalaire sur $L$, et le fait que $v\notin 2L$ (comme $v$ engendre $x$).

Il suffit ensuite de constater que l'image de la famille $\left\{(a_j\cdot v)\, a_i - (a_i\cdot v)\, a_j \vert i\in \{1,\dots ,n\}\right\}$ par la projection $L\rightarrow L/2L$ engendre bien bien $x^\perp$ (et donc que $\{2\, a_i \vert i\in \{1,\dots,n\} \} \cup \left\{(a_j\cdot v)\, a_i - (a_i\cdot v)\, a_j \vert i\in \{1,\dots ,n\}\right\} $ engendre bien $M$), et que le vecteur $v'=v-\frac{(v\cdot v)}{2}a_j$ satisfait bien les conditions voulues, ce que l'on vérifie facilement.
\end{proof}

\indent Donnons nous $\{a_1,\dots,a_n\}$ une base d'un réseau $L\in \mathcal{L}_n$, ainsi qu'une droite isotrope $x\in C_L(\F_2)$ engendré par un vecteur $v\in L$. Si l'on note $L_x$ le $2$-voisin de $L$ associé à $x$, alors l'algorithme qui donne $\Phi_n(L_x)$ se fait selon les étapes suivantes:

\textbf{Première étape :} grâce au lemme \ref{Lx}, on possède une famille $\Z$-génératrice de $L_x$.

\textbf{Deuxième étape :} la fonction {\rm{qflll}} de {\rm{PARI}} nous donne, à partir de la famille génératrice précédente, une base $\{b_1,\dots,b_n\}$ de $L_x$, ainsi que la matrice de Gram associée $\widetilde{B}=(b_i\cdot b_j)_{i,j}$.

\textbf{Troisième étape :} grâce à la base $\{b_i \}$ et à la matrice $\widetilde{B}$, la fonction {\rm{qfminim}} de {\rm{PARI}} nous donne l'ensemble $R$ des racines de $L_x$, ainsi qu'un système $R^+$ de racines positives.

\textbf{Quatrième étape :} en posant $\rho = \frac{1}{2}\sum_{\beta \in R^+} \beta$, on déduit le système de racines simples suivant $\{\beta_1 ,\dots ,\beta_k\} = \{\beta \in R^+ \vert (\rho\cdot \beta)=1 \}$, ainsi que la matrice de Cartan associée $B=(\beta_i \cdot \beta_j)_{i,j}$.

\textbf{Cinquième étape :} grâce au cardinal de $R$ ainsi qu'à la matrice $B$, on déduit $\Phi_n(L_x)$.

\subsection{Présentation de l'algorithme de calcul de la matrice de $\mathrm{T}_2$ sur $\Z[X_{n}]$.}\label{3.3}
On aura besoin de parcourir, pour $L\in \mathcal{L}_{n}$, tous les éléments de $C_L(\F_2)$, ce que l'on fera à l'aide d'une $\Z$-base $a_1,\dots ,a_n$ de $L$ et du lemme évident suivant :
\begin{lem} \label{Vn} Soient $L\in \mathcal{L}_n$, et $a_1,\dots,a_n$ une $\Z$-base arbitraire de $L$. On définit l'ensemble $V_n(a_1,\dots,a_n)=\left\{ (v_1,\dots,v_n)\in \{0,1\}^n \setminus \{0\} \, \vert \, q\left( \sum_k v_k\, a_k \right) \equiv 0\mathrm{\ mod\ }2 \right\}$. Alors l'application :
$$\begin{array}{rcl}
V_n(a_1,\dots,a_n) & \rightarrow & C_L(\F_2)\\
(v_1,\dots,v_n) & \mapsto & \mathrm{vect}_{\Z/2} (\sum_k v_k\, a_k)\\
\end{array},$$
est une bijection.
\end{lem}

\subsubsection{L'algorithme de calcul de la matrice de $\mathrm{T}_2$ sur $\Z[X_{23}]$.}
Notons $R_1,\dots,R_{32}$ les classes d'isomorphisme des systèmes de racines des éléments de $\mathcal{L}_{23}$ (suivant la numérotation de la table \ref{rX23}). Pour $i\in \{1,\dots,32\}$, on pose de plus $\phi_i$ l'image par $\Phi_{23}$ (introduit à la proposition \ref{Phi}) d'un réseau $L\in \mathcal{L}_{23}$ tel que $R(L)\simeq R_i$, et on note $\overline{L}_i$ la classe de $L$ dans $X_{23}$. On donne dans la table \ref{rX23} la numérotation choisie pour les éléments de $X_{23}$, en donnant les valeurs de $R_i$ et de $\phi_i$ en fonction de $i$.

On souhaite déterminer la matrice $T_{23}\in \mathrm{M}_{32}(\Z)$ de l'opérateur $\mathrm{T}_2$ sur $\Z[X_{23}]$ dans la base $(\overline{L}_1,\dots,\overline{L}_{32})$. Pour cela, on cherche pour tout $i\in \{1,\dots,32\}$ un réseau $L\in \mathcal{L}_{23}$ tel que $R(L)\simeq R_i$, et on le munit d'une base $a_1,\dots ,a_{23}$ (en pratique, il s'agira de la base fournie par la fonction {\rm{qflll}} de {\rm{PARI}}). On construit ensuite tous les éléments $x\in C_L(\F_2)$ (grâce au lemme \ref{Vn}), et pour chaque $x$ on détermine la classe d'isomorphisme du $2$-voisin $L_x$ de $L$ associé à $x$ (grâce à l'algorithme décrit au paragraphe \ref{3.2}).

Soulignons qu'il n'est pas évident pour tout $i$ de construire un réseau de la forme précédente. Cependant, la connexité du graphe de Kneser $\mathrm{K}_{23}(2)$ (d'après la proposition \ref{kneserconn}) nous dit qu'il suffit en fait d'avoir un seul réseau $L\in \mathcal{L}_{23}$. En effet, en parcourant les $2$-voisins de $L$, on parcourra des éléments de $\mathcal{L}_{23}$ correspondant à d'autres classes d'isomorphisme dans $X_{23}$. En répétant le processus, on parcourra toutes les classes d'isomorphisme de $X_{23}$, comme le graphe $\mathrm{K}_{23}(2)$ est connexe.

Le réseau que l'on a utilisé comme point de départ est donné par le lemme suivant :
\begin{lem} On se place dans $\R^{24}$ muni de sa base canonique $(e_i)$, et on considère le réseau $M$ engendré par : $\{e_i\pm e_{i+1} \,\vert\, 1\leq i\leq 21\}\cup \{e_{23}-e_{24}\}$.

Le réseau $L=M+\frac{1}{2}\Z\left( \sum_{i=1}^{24} e_i\right)$ est un élément de $\mathcal{L}_{23}$, et vérifie $R(L)\simeq \mathbf{D}_{22}\coprod \mathbf{A}_1$.
\end{lem}
\begin{proof}
On vérifie facilement que $L$ est un $\Z$-module libre de rang $23$, et que son déterminant est $2$. Son système de racine est : $R(L) = \{ \pm e_i \pm e_j \vert 1\leq i<j\leq 22 \} \cup \{ \pm (e_{23}-e_{24}) \}$, qui vérifie bien $R(L)\simeq \mathbf{D}_{22}\coprod \mathbf{A}_1$.
\end{proof}

Expliquons en détail l'algorithme du calcul de $\mathrm{T}_2$ sur $\Z[X_{23}]$. Pour $i\in \{1,\dots,32\}$, nous allons définir des réseaux $L_i\in \mathcal{L}_n$ satisfaisant $R(L_i)\simeq R_i$ (suivant la numérotation de la table \ref{rX23}), munis d'une base arbitraire $a_{i,1},\dots,a_{i,23}$ (en pratique, il s'agira de la base fournie par la fonction {\rm{qflll}} de {\rm{PARI}}). Les coefficients de la matrice $T_{23}$ se déduiront des applications $t_i:V_{23}(a_{i,1},\dots,a_{i,23}) \rightarrow \{1,\dots,32\}$ définies ci-dessous.

On pose $L_1$ le réseau donné par le lemme précédent, et on pose $a_{1,1},\dots,a_{1,23}$ la base de $L_1$ fournie par {\rm{PARI}}. On définit comme suit la fonction $t_1 : V_{23}(a_{1,1},\dots,a_{1,23})\rightarrow \{1,\dots,32\}$.

Soit $(v_1,\dots,v_{23})\in V_{23}(a_{1,1},\dots,a_{1,23})$. On pose $x=\mathrm{vect}_{\Z/2}(\sum_k v_k \, a_{1,k})\in C_{L_1}(\F_2)$. On reprend la construction du $2$-voisin $L_x$ de $L_1$ associé à $x$ dans la proposition \ref{Lx} : il existe un unique $i\in \{1,\dots ,32\}$ tel que $R(L_x)\simeq R_i$ (à savoir l'unique $i$ tel que $\Phi_{23}(L_x)=\phi_i$). On pose alors : $t_1((v_1,\cdots,v_{23}))=i$.

Pour tout $i\in t_1(V_{23}(a_{1,1},\dots,a_{1,23}))$, on fixe un élément $(v_1,\dots,v_{23})\in t_1^{-1}(\{i\})$ quelconque. Notons comme précédemment $x=\mathrm{vect}_{Z/2}(\sum_k v_k a_{1,k})$ et $L_x$ le $2$-voisin de $L_1$ associé à $x$ : on pose $a_{i,1},\dots,a_{i,23}$ la base de $L_x$ fournie par {\rm{PARI}}, et on définit le réseau $L_i=L_x$.

Pour tous les réseaux $L_i\in \mathcal{L}_{23}$ ainsi définis, muni de leur base $a_{i,1},\dots,a_{i,23}$, on construit de même que pour $L_1$ la fonction $t_i : V_{23}(a_{i,1},\dots,a_{i,23})\rightarrow \{1,\dots,32\}$. De plus, pour $j\in t_i(V_{23}(a_{i,1},\dots,a_{i,23}))$, si l'on n'a pas précédemment défini de réseau $L_j$, on choisit un élément quelconque $(v_1,\dots,v_{23})\in t_i^{-1}(\{j\})$ : on note à nouveau $x=\mathrm{vect}_{Z/2}(\sum_k v_k a_{i,k})$ et $L_x$ le $2$-voisin de $L_i$ associé à $x$, et on pose $a_{j,1},\dots,a_{j,23}$ la base de $L_x$ fournie par {\rm{PARI}}, ainsi que $L_j=L_x$.

On répète le processus jusqu'à avoir construit pour tout $i\in \{1,\dots,32\}$ un réseau $L_i\in \mathcal{L}_{23}$ de base $a_{i,1},\dots,a_{i,23}$ tel que $\Phi_{23}(L_i)=\phi_i$, ainsi qu'une fonction $t_i : V_{23}(a_{i,1},\dots,a_{i,23})\rightarrow \{1,\dots,32\}$. 

Cet algorithme se termine bien, comme le graphe de Kneser $\mathrm{K}_{23}(2)$ est connexe d'après la proposition \ref{kneserconn}.

\begin{prop} On pose $T_{23}=(t_{i,j})\in \mathrm{M}_{32}(\Z)$, c'est-à-dire que $t_{i,j}$ est le nombre de $2$-voisins de $L_j$ isomorphes à $L_i$. Avec les notations précédentes, on a l'égalité :
$$t_{i,j} = \vert t_j^{-1}(\{i\})\vert .$$
\end{prop}
\begin{proof}
Soient $j\in \{1,\dots ,32\}$ et $(v_1,\dots,v_{23})\in V_{23}(a_{j,1},\dots,a_{j,23})$. On pose comme précédemment $v=\sum_k v_k\, a_{j,k}$, $x=\mathrm{vect}_{Z/2}(v)$ et $L_x$ le $2$-voisin de $L_j$ associé à $x$. On a les équivalence suivantes :
$$L_x\simeq L_i\ \Leftrightarrow \ R(L_x)\simeq R(L_i)\simeq R_i\ \Leftrightarrow \ \Phi_{23}(L_x) = \phi_i \ \Leftrightarrow \ t_j((v_1,\dots,v_{23}))=i$$
et donc :
$$t_{i,j}=\big\vert \{x\in C_{L_j}(\F_2) \vert L_x\simeq L_i \} \big\vert = \vert t_j^{-1}(\{i\})\vert $$
qui est l'égalité cherchée.
\end{proof}

On renvoie à \cite{MegT2X23} pour un algorithme détaillé du calcul de $T_{23}$. Dans cet algorithme, les fichiers \textit{``generateursX23Li"} sont situés dans le dossier parent : ils contiennent des bases des réseaux $L_i$ que l'on a utilisés pour notre algorithme.

\subsubsection{L'algorithme de calcul de la matrice de $\mathrm{T}_2$ sur $\Z[X_{25}]$.}
On reprend la liste de \cite{Bor} pour l'ensemble des classes de $X_{25}$, et on note $R_1,\dots,R_{121}$ les classes d'isomorphisme des systèmes de racines des éléments de $\mathcal{L}_{25}$ (suivant la numérotation de \cite{Bor}, reprise dans la table \ref{rX25}). Comme dans le cas de $X_{23}$, pour $i\in \{1,\dots,121\}$, on pose de plus $\phi_i$ l'image par $\Phi_{25}$ (introduit à la proposition \ref{Phi}) d'un réseau $L\in \mathcal{L}_{25}$ tel que $R(L)\simeq R_i$, et on note $\overline{L}_i$ la classe de $L$ dans $X_{25}$. On donne dans la table \ref{rX25} la numérotation choisie pour les éléments de $X_{25}$, en donnant les valeurs de $R_i$ et de $\phi_i$ en fonction de $i$.

On souhaite déterminer la matrice $T_{25}\in \mathrm{M}_{121}(\Z)$ de l'opérateur $\mathrm{T}_2$ sur $\Z[X_{25}]$ dans la base $(\overline{L}_1,\dots,\overline{L}_{121})$. Pour cela, on cherche pour tout $i\in \{1,\dots,121\}$ un réseau $L\in \mathcal{L}_{25}$ tel que $R(L)\simeq R_i$, et on le munit d'une base $a_1,\dots ,a_{25}$ (en pratique, il s'agira de la base fournie par la fonction {\rm{qflll}} de {\rm{qflll}}). On construit ensuite tous les éléments $x\in C_L(\F_2)$ (grâce au lemme \ref{Vn}), et pour chaque $x$ on détermine la classe d'isomorphisme du $2$-voisin $L_x$ de $L$ associé à $x$ (grâce à l'algorithme décrit au paragraphe \ref{3.2}).

Là encore il n'est pas évident pour tout $i$ de construire un réseau de la forme précédente. De la même manière que dans le cas de $X_{23}$, on utilise la connexité du graphe de Kneser $\mathrm{K}_{25}(2)$ (d'après la proposition \ref{kneserconn}). Suivant le raisonnement adopté précédemment, il suffit de construire un seul réseau $L\in \mathcal{L}_{25}$, puis de construire ses $2$-voisins, et de répéter le processus jusqu'à avoir parcouru tous les éléments de $X_{25}$.

Le réseau que l'on a utilisé comme point de départ est donné par le lemme suivant :
\begin{lem} On se place dans $\R^{26}$ muni de sa base canonique $(e_i)$, et on considère le réseau $M$ engendré par : $\{e_i\pm e_{i+1} \,\vert\, 1\leq i\leq 23\}\cup \{e_{25}-e_{26}\}$.

Le réseau $L=M+\frac{1}{2}\Z\left( \sum_{i=1}^{26} e_i\right)$ est un élément de $\mathcal{L}_{25}$, et vérifie $R(L)\simeq \mathbf{D}_{24}\coprod \mathbf{A}_1$.
\end{lem}

\begin{proof}
On vérifie facilement que $L$ est un $\Z$-module libre de rang $25$, et que son déterminant est $2$. Son système de racine est : $R(L) = \{ \pm e_i \pm e_j \vert 1\leq i<j\leq 24 \} \cup \{ \pm (e_{25}-e_{26}) \}$, qui vérifie bien $R(L)\simeq \mathbf{D}_{24}\coprod \mathbf{A}_1$.
\end{proof}

L'algorithme du calcul de $\mathrm{T}_2$ sur $\Z[X_{25}]$ suit le même processus que dans le cas de $X_{23}$. Pour $i\in \{1,\dots,121\}$, nous allons définir des réseaux $L_i\in \mathcal{L}_n$ satisfaisant $R(L_i)\simeq R_i$ (suivant la numérotation de la table \ref{rX25}), munis d'une base arbitraire $a_{i,1},\dots,a_{i,25}$ (en pratique, il s'agira de la base fournie par la fonction {\rm{qflll}} de {\rm{PARI}}). Les coefficients de la matrice $T_{25}$ se déduiront des applications $t_i:V_{25}(a_{i,1},\dots,a_{i,25}) \rightarrow \{1,\dots,121\}$ définies ci-dessous.

On pose $L_{121}$ le réseau donné par le lemme précédent, et on pose $a_{{121},1},\dots,a_{{121},25}$ la base de $L_{121}$ fournie par {\rm{PARI}}. On définit comme suit la fonction $t_{121} : V_{25}(a_{{121},1},\dots,a_{{121},25})\rightarrow \{1,\dots,121\}$.

Soit $(v_1,\dots,v_{25})\in V_{25}(a_{{121},1},\dots,a_{{121},25})$. On pose $x=\mathrm{vect}_{\Z/2}(\sum_k v_k \, a_{{121},k})\in C_{L_{121}}(\F_2)$. On reprend la construction du $2$-voisin $L_x$ de $L_{121}$ associé à $x$ dans la proposition \ref{Lx} : il existe un unique $i\in \{1,\dots ,121\}$ tel que $R(L_x)\simeq R_i$ (à savoir l'unique $i$ tel que $\Phi_{25}(L_x)=\phi_i$). On pose alors : $t_{121}((v_1,\cdots,v_{25}))=i$.

Pour tout $i\in t_1(V_{25}(a_{{121},1},\dots,a_{{121},25}))$, on fixe un élément $(v_1,\dots,v_{25})\in t_{121}^{-1}(\{i\})$ quelconque. Notons comme précédemment $x=\mathrm{vect}_{\Z/2}(\sum_k v_k a_{{121},k})$ et $L_x$ le $2$-voisin de $L_{121}$ associé à $x$ : on pose $a_{i,1},\dots,a_{i,25}$ la base de $L_x$ fournie par {\rm{PARI}}, et on définit le réseau $L_i=L_x$.

Pour tous les réseaux $L_i\in \mathcal{L}_{25}$ ainsi définis, muni de leur base $a_{i,1},\dots,a_{i,25}$, on construit de même que pour $L_{121}$ la fonction $t_i : V_{25}(a_{i,1},\dots,a_{i,25})\rightarrow \{1,\dots,121\}$. De plus, pour $j\in t_i(V_{25}(a_{i,1},\dots,a_{i,25}))$, si l'on n'a pas précédemment défini de réseau $L_j$, on choisit un élément quelconque $(v_1,\dots,v_{25})\in t_i^{-1}(\{j\})$ : on note à nouveau $x=\mathrm{vect}_{Z/2}(\sum_k v_k a_{i,k})$ et $L_x$ le $2$-voisin de $L_i$ associé à $x$, et on pose $a_{j,1},\dots,a_{j,25}$ la base de $L_x$ fournie par {\rm{PARI}}, ainsi que $L_j=L_x$.

On répète le processus jusqu'à avoir construit pour tout $i\in \{1,\dots,121\}$ un réseau $L_i\in \mathcal{L}_{25}$ de base $a_{i,1},\dots,a_{i,25}$ tel que $\Phi_{25}(L_i)=\phi_i$, ainsi qu'une fonction $t_i : V_{25}(a_{i,1},\dots,a_{i,25})\rightarrow \{1,\dots,121\}$. 

Cet algorithme se termine bien, comme le graphe de Kneser $\mathrm{K}_{25}(2)$ est connexe d'après la proposition \ref{kneserconn}.
\begin{prop} On pose $T_{25}=(t_{i,j})\in \mathrm{M}_{121}(\Z)$, c'est-à-dire que $t_{i,j}$ est le nombre de $2$-voisins de $L_j$ isomorphes à $L_i$. Avec les notations précédentes, on a l'égalité :
$$t_{i,j} = \vert t_j^{-1}(\{i\})\vert .$$
\end{prop}
\begin{proof}
Soient $j\in \{1,\dots ,121\}$ et $(v_1,\dots,v_{25})\in V_{25}(a_{j,1},\dots,a_{j,25})$. On pose comme précédemment $v=\sum_k v_k\, a_{j,k}$, $x=\mathrm{vect}_{Z/2}(v)$ et $L_x$ le $2$-voisin de $L_j$ associé à $x$. On a les équivalence suivantes :
$$L_x\simeq L_i\ \Leftrightarrow \ R(L_x)\simeq R(L_i)\simeq R_i\ \Leftrightarrow \ \Phi_{25}(L_x) = \phi_i \ \Leftrightarrow \ t_j((v_1,\dots,v_{25}))=i$$
et donc :
$$t_{i,j}=\big\vert \{x\in C_{L_j}(\F_2) \vert L_x\simeq L_i \} \big\vert = \vert t_j^{-1}(\{i\})\vert $$
qui est l'égalité cherchée.
\end{proof}

On renvoie à \cite{MegT2X25} pour un algorithme détaillé du calcul de $T_{25}$. Dans cet algorithme, les fichiers \textit{``generateursX25Li"} sont situés dans le dossier parent : ils contiennent des bases des réseaux $L_i$ que l'on a utilisés pour notre algorithme.

\subsection{Résultats obtenus.}\label{3.4}
Les résultats obtenus grâce à nos algorithmes sont donnés par les théorèmes suivants :
\begin{theo} Soient $L_1,\dots,L_{32}\in \mathcal{L}_{23}$ vérifiant pour tout $i$ : $R(L_i)\simeq R_i$ (suivant les notations de la table \ref{rX23}), et soient $\overline{L}_1,\dots,\overline{L}_{32}$ leurs classes respectives dans $X_{23}$. Alors la matrice de l'opérateur $\mathrm{T}_2$ relativement à la base $\overline{L}_1,\dots,\overline{L}_{32}$ est donnée dans \cite{MegTp}.
\end{theo}
\begin{theo} Soient $L_1,\dots,L_{121}\in \mathcal{L}_{25}$ vérifiant pour tout $i$ : $R(L_i)\simeq R_i$ (suivant les notations de la table \ref{rX25}), et soient $\overline{L}_1,\dots,\overline{L}_{121}$ leurs classes respectives dans $X_{25}$. Alors la matrice de l'opérateur $\mathrm{T}_2$ relativement à la base $\overline{L}_1,\dots,\overline{L}_{121}$ est donnée dans \cite{MegTp}.
\end{theo}
\section{Applications.}\label{4}
\subsection{La codiagonalisation des matrices opérateurs $\mathrm{T}_p$.}\label{4.1}
Pour $n=23$ ou $25$, et $p$ premier, on note $T_n(p)$ la matrice de l'opérateur $\mathrm{T}_p$ sur $X_n$ relativement à la base des $\overline{L}_i$ introduite précédemment (suivant les numérotations des tables \ref{rX23} ou \ref{rX25}, selon la valeur de $n$). En particulier, on a : $T_{23}=T_{23}(2)$ et $T_{25}=T_{25}(2)$ (suivant les notations du paragraphe \ref{3.3}).

Notre point de départ est la proposition suivante :

\begin{prop}\label{codiag} Pour $n=23$ ou $n=25$, les opérateurs $\mathrm{T}_p:\C[X_n]\rightarrow \C[X_n]$ pour $p$ premier sont codiagonalisables. De manière équivalente, une fois $n$ fixé, les matrices $T_n(p)$ pour $p$ premier sont codiagonalisables dans $\C$.
\end{prop}
\begin{proof}
C'est un résultat classique, que l'on retrouve notamment dans \cite{CL} à de nombreuses reprises. Pour s'en convaincre dans notre cas, il suffit de constater que les matrices $T_{23}$ et $T_{25}$ calculées au chapitre précédent ont toutes leurs valeurs propres distinctes, et que les anneaux de Hecke $\mathrm{H}(\mathrm{O}_n)$ sont commutatifs.

On en déduit qu'on a même un résultat plus fort : si $n=23$ ou $25$, tous les éléments de $\mathrm{H}(\mathrm{O}_n)$ sont codiagonalisables, et toute base de diagonalisation de $\mathrm{T}_2$ est un base de codiagonalisation de $\mathrm{H}(\mathrm{O}_n)$.
\end{proof}

Grâce à {\rm{PARI}}, il est facile de trouver des bases de diagonalisation des matrices $T_{23}$ et $T_{25}$ que l'on a calculées, qui seront nécessairement des bases de codiagonalisation respectivement pour $\mathrm{H}(\mathrm{O}_{23})$ et $\mathrm{H}(\mathrm{O}_{25})$.

Notons respectivement $v_1,\dots,v_{32}$ et $w_1,\dots,w_{121}$ les bases de diagonalisation obtenues pour $T_{23}$ et $T_{25}$, en supposant que les valeurs propres  $\lambda_i$ ou $\mu_i$ correspondantes suivent les numérotations des tables \ref{psi23}, \ref{psi25} et \ref{psi252}. Du fait des valeurs que prennent les $\lambda_i$ et les $\mu_i$, on peut choisir les $v_i$ et les $w_i$ de telle sorte que :

- pour $i\in \{1,\dots,32\}$, les $v_i$ sont dans $\Z^{32}$, et de coordonnées premières entre elles ;

- pour $i\in \{1,\dots,57\}$, les $w_i$ sont dans $\Z^{121}$, et de coordonnées premières entre elles ;

- pour $i\in \{58,\dots,121\}$, les $w_i$ sont dans $\Z[\sqrt{144169}]^{121}$, et de coordonnées premières entre elles.

Les valeurs propres de matrices $T_{23}(p)$ et $T_{25}(p)$ associées respectivement aux vecteurs $v_i$ et $w_i$ sont alors données par les propositions suivantes :

\begin{prop} \label{lambdai} En reprenant les notations précédentes, les vecteurs $v_i\in \Z^{32}$ constituent des vecteurs propres communs à tous les éléments de $\mathrm{H}(\mathrm{O}_{23})$. De plus, chacun de ces vecteurs $v_i$ engendre une représentation automorphe $\pi_i\in \Pi_{\mathrm{cusp}}(\mathrm{SO}_{23})$ dont le paramètre standard $\psi_i=\psi(\pi_i,\mathrm{St})$ est donné par la table \ref{psi23}.

Pour $p$ un nombre premier, notons $\lambda_i(p)$ la valeur propre pour $\mathrm{T}_p$ associée au vecteur $v_i$ (en particulier, $\lambda_i(2)=\lambda_i$). Alors on a la formule :
$$\lambda_i(p)=p^{\frac{21}{2}}\mathrm{Trace}\left( c_p(\pi_i) \vert V_{\mathrm{St}} \right).$$
\end{prop}

\begin{proof}
Considérons $p$ un nombre premier, et plaçons nous dans l'anneau de Hecke $\mathrm{H}(\mathrm{O}_{23})$. Reprenons la $\Q$-base $v_1,\dots,v_{32}$ de l'espace vectoriel $\Q[X_{23}]\simeq \Q^{32}$ définie précédemment. Chacun des vecteurs $v_i$ engendre une représentation automorphe $\pi_i\in \Pi_{\mathrm{cusp}}(\mathrm{SO}_{23})$. Les paramètres standards $\psi(\pi_i,\mathrm{St})$ de tels $\pi_i$ sont donnés par \cite[Table C.7]{CL}.

Les formules de Gross (traitées dans \cite{Gr} dans le cas des groupes adjoints, et précisées dans \cite[ch. VI, lemme 2.7]{CL} dans le cas des groupes semi-simples) nous donnent la relation suivante : $p^{\frac{21}{2}}\left[ V_{\mathrm{St}}\right] = \mathrm{T}_p$. On en déduit la relation cherchée entre $\lambda_i(p)$ et $\pi_i$ pour tout $p$.

Il suffit ensuite de vérifier la relation $\lambda_i=\lambda_i(2)$ pour s'assurer que l'indexation des $\pi_i$ correspond bien à celle choisie pour les $v_i$ (ce qui a bien un sens, comme tous les $\lambda_i$ sont distincts).
\end{proof}

\begin{prop} \label{mui} En reprenant les notations précédentes, les vecteurs $w_i\in \Z[\sqrt{144169}]^{121}$ constituent des vecteurs propres communs à tous les éléments de $\mathrm{H}(\mathrm{O}_{25})$. De plus, chacun de ces vecteurs $w_i$ engendre une représentation automorphe $\pi_i\in \Pi_{\mathrm{cusp}}(\mathrm{SO}_{25})$ dont le paramètre standard $\psi_i=\psi(\pi_i,\mathrm{St})$ est donné par les tables \ref{psi25} et \ref{psi252}.

Pour $p$ un nombre premier, notons $\mu_i(p)$ la valeur propre pour $\mathrm{T}_p$ associée au vecteur $w_i$ (en particulier, $\mu_i(2)=\mu_i$). Alors on a la formule :
$$\mu_i(p)=p^{\frac{23}{2}}\mathrm{Trace}\left( c_p(\pi_i) \vert V_{\mathrm{St}} \right).$$
\end{prop}
\begin{proof}
La démonstration se fait comme précédemment. Considérons $p$ un nombre premier, et plaçons nous dans l'anneau de Hecke $\mathrm{H}(\mathrm{O}_{25})$. Reprenons la $\Q[\sqrt{144169}]$-base $w_1,\dots,w_{121}$ de l'espace vectoriel $\Q[\sqrt{144169}][X_{25}]\simeq \Q[\sqrt{144169}]^{121}$ définie précédemment. Chacun des vecteurs $w_i$ engendre une représentation automorphe $\pi_i\in \Pi_{\mathrm{cusp}}(\mathrm{SO}_{25})$. Les paramètres standards $\psi(\pi_i,\mathrm{St})$ de tels $\pi_i$ sont donnés par \cite[Appendix D]{CR}.

Les formules de Gross nous donnent la relation suivante : $p^{\frac{23}{2}}\left[ V_{\mathrm{St}}\right] = \mathrm{T}_p$. On en déduit la relation cherchée entre $\mu_i(p)$ et $\pi_i$ pour tout $p$.

Il suffit ensuite de vérifier la relation $\mu_i=\mu_i(2)$ pour vérifier que l'indexation des $\pi_i$ correspond bien à celle choisie pour les $w_i$ (ce qui a encore bien un sens, comme tous les $\mu_i$ sont distincts).
\end{proof}

Afin de simplifier les notations, on définit les matrices $V\in \mathrm{M}_{32}(\R)$ et $W\in \mathrm{M}_{121}(\R)$ dont les colonnes sont respectivement les vecteurs $v_i$ et $w_i$. Pour $(a_1,\dots,a_n)\in\R^n$, on note $\mathrm{diag}(a_1,\dots,a_n)\in\mathrm{M}_n(\R)$ la matrice diagonale dont les coefficients diagonaux sont les $a_i$. Alors on a les égalités suivantes :
$$
\begin{aligned}
T_{23}(p) & = V\, \mathrm{diag}(\lambda_1(p),\dots,\lambda_{32}(p))\, V^{-1} ; \\
T_{25}(p) & = W\, \mathrm{diag}(\mu_1(p),\dots,\mu_{121}(p))\, W^{-1} .
\end{aligned}
$$

En particulier, on a la proposition suivante :
\begin{theo} \label{Tp} Pour tout $p\leq 113$ premier, les matrices $T_{23}(p)$ sont données dans \cite{MegTp}. Pour tout $p\leq 67$ premier, les matrices $T_{25}(p)$ sont données dans \cite{MegTp}.
\end{theo}
\begin{proof}
Il est équivalent de connaître les matrices $T_{23}(p)$ et $T_{25}(p)$ et de connaître les $\lambda_i(p)$ et les $\mu_j(p)$. D'après les propositions \ref{lambdai} et \ref{mui}, il suffit donc de connaître les valeurs de $\mathrm{Trace}(c_p(\pi)\vert V_{\mathrm{St}})$, pour les éléments $\pi\in \cup_m \Pi_{\mathrm{alg}}^\bot(\mathrm{PGL}_m)$ apparaissant dans les tables \ref{psi23}, \ref{psi25} et \ref{psi252}.

Les calculs de ces $\mathrm{Trace}(c_p(\pi)\vert V_{\mathrm{St}})$ ont justement été effectués jusqu'à $p\leq 67$ dans tous ces cas (voir \cite{Meg}). Les résultats de \cite{CL} permettent même d'aller plus loin, et de connaître pour tout $p\leq 113$ la matrice $\mathrm{T}_{23}(p)$ (grâce aux résultats sur les formes modulaires de Siegel de genre $2$).
\end{proof}

Les matrices $T_{23}(p)$ (pour $p\leq 113$) et $T_{25}(p)$ (pour $p\leq 67$) nous donnent en particulier les corollaires suivants :

\begin{cor} Soit $p$ un nombre premier. Le diamètre du graphe $\mathrm{K}_{23}(p)$ est le suivant : $4$ pour $p=2$, $3$ pour $p=3$, $2$ pour $5\leq p \leq 19$, et $1$ pour $23\leq p\leq 113$.
\end{cor}

\begin{cor} Soit $p$ un nombre premier. Le diamètre du graphe $\mathrm{K}_{25}(p)$ est le suivant : $6$ pour $p=2$, $4$ pour $p=3$, $3$ pour $5\leq p \leq 7$, et $2$ pour $11\leq p\leq 61$.
\end{cor}

Pour les plus grandes valeurs de $p$, on a le théorème suivant :

\begin{theo} Soit $p$ un nombre premier. Si $p\geq 23$, le graphe $\mathrm{K}_{23}(p)$ est complet. Si $p\geq 67$, le graphe $\mathrm{K}_{25}(p)$ est complet.
\end{theo}
\begin{proof}
La démonstration suit celle de \cite[ch. X, Théorème 2.4]{CL}. Nous la reprenons en détail dans le cas de $\mathrm{K}_{23}(p)$, et ne donnerons que quelques points clefs pour le cas de $\mathrm{K}_{25}(p)$.

\`A la manière de \cite[ch. X, \S2]{CL}, définissons les fonctions $\theta_1(p)=D_{11}(p)$, $\theta_2(p)=D_{15}(p)$, $\theta_3(p)=D_{17}(p)$, $\theta_4(p)=D_{19}(p)$, $\theta_5(p)=D_{21}(p)$, $\theta_6(p)=D_{19,7}(p)$, $\theta_7(p)=D_{21,5}(p)$,  $\theta_8(p)=D_{21,9}(p)$ et  $\theta_9(p)=D_{21,13}(p)$. D'après la proposition \ref{lambdai}, il existe des polynômes $C_{i,r}\in \Z[X]$, pour $1\leq i\leq 32$ et $0\leq r\leq 9$, uniquement déterminés, tels que l'on a pour tout $1\leq i\leq 32$ et pour tout $p$ premier :
$$\lambda_i(p) = C_{i,0}(p)+\sum_{r=1}^9 C_{i,r}(p)\theta_r(p).$$

Si l'on note $T_{23}(p) = (t_{i,j}(p))_{1\leq i,j\leq 32}$, alors il existe des polynômes $P_{i,j,r}\in \Q[X]$, pour $1\leq i,j\leq 32$ et $0\leq r\leq 9$, uniquement déterminés, tels que l'on a pour tout $1\leq i,j\leq 32$ et pour tout $p$ premier :
$$t_{i,j}(p) = P_{i,j,0}(p) + \sum_{r=1}^9 P_{i,j,r}(p) \theta_r(p).$$

Supposons que les deux sommets de $\mathrm{K}_{23}(p)$ correspondant aux classes $\overline{L_i},\overline{L_j}\in X_{23}$ ne sont pas connexes. Ceci est équivalent à dire que $t_{i,j}(p)=0$, et on a alors :
$$P_{i,j,0}(p)^2 = \left( \sum_{r=1}^9 P_{i,j,r}(p) \theta_r(p)\right)^2.$$

Par l'inégalité de Cauchy-Schwarz, a
$$\left( \sum_{r=1}^9 P_{i,j,r}(p) \theta_r(p)\right)^2\leq \left( \sum_{r=1}^9 P_{i,j,r}(p)^2 \gamma_r\right) \left( \sum_{r=1}^9 \gamma_r^{-1} \theta_r(p)^2\right)$$
pour tout $9$-uple $(\gamma_1,\dots,\gamma_9)$ de réels strictement positifs. En particulier, grâce aux inégalités de Ramanujan exposées à la proposition \ref{rama}, et en prenant $(\gamma_1,\dots,\gamma_9)=(4p^{11},4p^{15},4p^{17},\newline 4p^{19},4p^{21},16p^{19},16p^{21},16p^{21},16p^{21})$, on déduit l'inégalité :
$$\left( \sum_{r=1}^9 P_{i,j,r}(p) \theta_r(p)\right)^2\leq 9\left( \sum_{r=1}^9 P_{i,j,r}(p)^2 \gamma_r\right). $$

On définit les polynômes
$$(\Gamma_1(X),\Gamma_2(X),\dots,\Gamma_9(X)) = (4X^{11},4X^{15},\dots,16X^{21})$$
et 
$$Q_{i,j}(X) = P_{i,j,0}(X)^2 - 9\left( \sum_{r=1}^9 P_{i,j,r}(X)^2 \Gamma_r(X)\right).$$

Les polynômes $Q_{i,j}(X)$ sont des éléments de $\Q[X]$ de coefficient dominant strictement positif. Si l'on définit $\rho_{i,j}$ comme la plus grande racine réelle de $Q_{i,j}(X)$ (avec la convention $\rho_{i,j}=-\infty$ si $Q_{i,j}$ n'a pas de racine réelle). Avec ces notations, on a $t_{i,j}(p)>0$ dès que $p> \rho_{i,j}$.

Il suffit ensuite de constater que $\mathrm{max}\{ \rho_{i,j}, 1\leq i,j\leq 32\} \approx 21.15$. Ceci conclut que pour $p\geq 23$, tous les $t_{i,j}(p)$ sont non nuls, et que le graphe $\mathrm{K}_{23}(p)$ est complet pour de tels $p$.\\

Pour le graphe $\mathrm{K}_{25}(p)$, on procède de la même manière. Les fonctions $\theta_r(p)$ sont les suivantes : $\theta_1(p)=D_{11}(p)$, $\theta_2(p)=D_{15}(p)$, $\theta_3(p)=D_{17}(p)$, $\theta_4(p)=D_{19}(p)$, $\theta_5(p)=D_{21}(p)$, $\{\theta_6(p),\theta_7(p)\}=D^2_{23}(p)$, $\theta_8(p)=p^{11}\mathrm{tr}\left( \mathrm{Sym}^2 \, c_p (\Delta_{11})\vert V_{\mathrm{St}} \right)=\left(D_{11}(p)\right) ^2+p^{11}$, $\theta_{9}(p)=D_{19,7}(p)$, $\theta_{10}(p)=D_{21,5}(p)$, $\theta_{11}(p)=D_{21,9}(p)$, $\theta_{12}(p)=D_{21,13}(p)$, $\theta_{13}(p)=D_{23,7}(p)$, $\theta_{14}(p)=D_{23,9}(p)$, $\theta_{15}(p)=D_{23,13}(p)$, $\theta_{16}(p)=D_{23,13,5}(p)$, $\theta_{17}(p)=D_{23,15,3}(p)$, $\theta_{18}(p)=D_{23,15,7}(p)$, $\theta_{19}(p)=D_{23,17,5}(p)$, $\theta_{20}(p)=D_{23,17,9}(p)$, $\theta_{21}(p)=D_{23,19,3}(p)$ et $\theta_{22}(p)=D_{23,19,11}(p)$.

Les polynômes $\Gamma_r(X)$ (qui nous donnent les quantités $\gamma_r=\Gamma_r(p)$) sont les suivantes : $\Gamma_1(X)=4X^{11}$, $\Gamma_2(X)=4X^{15}$, $\Gamma_3(X)=4X^{17}$, $\Gamma_4(X)=4X^{19}$, $\Gamma_5(X)=4X^{21}$, $\Gamma_6(X)=\Gamma_7(X)=4X^{23}$, $\Gamma_8(X)=9X^{22}$, $\Gamma_9(X)=16X^{19}$, $\Gamma_{10}(X)=\dots=\Gamma_{12}(X)=16X^{21}$, $\Gamma_{13}(X)=\dots=\Gamma_{15}(X)=16X^{23}$ et $\Gamma_{16}(X)=\dots=\Gamma_{22}(X)=36X^{23}$.

On calcule de la même manières les quantités $\rho_{i,j}$ associées au polynômes $Q_{i,j}(X)$. Il suffit ensuite de constater que $\mathrm{max}\{ \rho_{i,j}, 1\leq i,j\leq 121\} \approx 64.25$. Ceci conclut que pour $p\geq 67$, tous les coefficients de la matrice $T_{25}(p)$ sont non nuls, et que le graphe $\mathrm{K}_{25}(p)$ est complet pour de tels $p$.
\end{proof}
\subsection{Quelques vérifications de nos résultats.}\label{4.2}
Les propositions \ref{lambdai} et \ref{mui} constituent une première vérification de nos calculs. En effet, on connaît grâce à ces propositions les valeurs propres de $\mathrm{T}_2$ sur $\Z[X_{23}]$ et $\Z[X_{25}]$ : on vérifie que ces valeurs correspondent bien aux valeurs propres des matrices $T_{23}$ et $T_{25}$ calculées par nos algorithmes.

Une deuxième vérification repose sur la remarque finale de \cite[Annexe B, \S 5.3]{CL}. Suivant l'indexation adoptée ici pour $X_{23}$, on a $L_2\simeq \mathrm{E}_{15}\oplus \mathrm{E}_8$ et $L_3\simeq \mathrm{E}_{16}\oplus \mathrm{E}_7$. Ainsi, le coefficient d'indice $(3,2)$ de la matrice $T_{23}$ est égal à $\mathrm{N}_2(\mathrm{E}_{15}\oplus \mathrm{E}_8,\mathrm{E}_{16}\oplus \mathrm{E}_7)$. On constate sur la matrice $T_{23}$ que l'on a calculé que ce coefficient vaut $120$, ce qui est bien cohérent avec le résultat de \cite{CL} évoqué ci-dessus.\\

Une autre vérification repose sur l'étude des graphes de Kneser $\mathrm{K}_{23}(p)$ ($2\leq p\leq 113$) et $\mathrm{K}_{25}(p)$ ($2\leq p\leq 67$), qu'on a calculés explicitement grâce au théorème \ref{Tp}. On vérifie dans un premier temps qu'ils sont connexes, comme exposé précédemment, ce qui est cohérent avec la proposition \ref{kneserconn}.

De plus, les graphes $\mathrm{K}_{24}(p)$ ont été calculé dans \cite[ch. X, Théorème 2.4]{CL} pour tout $p$ premier. On vérifie que nos calculs des graphes $\mathrm{K}_{23}(p)$ sont cohérents avec la propriété suivante :
\begin{prop} Soient $n\equiv -1\mathrm{\ mod\ }8$ et $p$ un nombre premier. Pour $i=1,2$, on se donne $L_i\in \mathcal{L}_n$ d'image $\overline{L_i}$ dans $X_n$, et $P_i\in \mathcal{L}_{n+1}$ d'image $\overline{P_i}\in X_{n+1}$ tel que $L_i\simeq P_i\cap \alpha_i^\perp$ pour un certain $\alpha_i\in R(P_i)$.

Si les sommets $\overline{L_1}$ et $\overline{L_2}$ sont adjacents dans le graphe $\mathrm{K}_n(p)$, alors les sommets $\overline{P_1}$ et $\overline{P_2}$ sont adjacents dans le graphe $\mathrm{K}_{n+1}(p)$.
\end{prop}
\begin{proof}
Considérons $L_1,L_2\in \mathcal{L}_n$ deux $p$-voisins, que l'on voit comme des réseaux de $L_1\otimes \Q = L_2\otimes \Q=\mathrm{U}_n$. Pour $i=1,2$, on pose $\widetilde{P}_i=L_i\oplus \mathrm{A}_1$, et $P_i$ l'unique réseau unimodulaire pair contenant $\widetilde{P}_i$ (c'est-à-dire l'image réciproque par $\widetilde{P}_i^\sharp \rightarrow \mathrm{r\acute{e}s}\, \widetilde{P}_i$ de l'unique droite isotrope de $\mathrm{r\acute{e}s}\, \widetilde{P}_i$).

Les réseaux $P_1,P_2$ sont des réseaux unimodulaires pairs de $\mathrm{V}_n=\mathrm{U}_n\oplus (\Q\otimes \mathrm{A}_1)$ qui vérifient pour $i=1,2$ : $L_i=P_i\cap \mathrm{U}_n$. On est ainsi dans le cadre de \cite[Annexe B, Proposition 4.2]{CL}. D'après cette proposition, les réseaux $P_1,P_2$ sont des $p$-voisins, ce qui conclut notre démonstration.
\end{proof}

On applique ce résultat aux graphes $\mathrm{K}_{23}(p)$ et $\mathrm{K}_{24}(p)$. On vérifie que cette propriété est satisfaite pour $p< 47$ (il est inutile de la vérifier pour $p\geq 47$ comme les graphes $\mathrm{K}_{23}(p)$ et $\mathrm{K}_{24}(p)$ sont alors complets). Pour cela, on utilise la table \ref{RX23}, qui nous donne pour chaque réseau $L\in \mathcal{L}_{23}$ la classe dans $X_{24}$ de l'unique réseau unimodulaire pair contenant $L\oplus \mathrm{A}_1$.

On vérifie aussi qu'on a le corollaire plus faible suivant :

\begin{cor} Soit $p$ un nombre premier. On note $\mathrm{K}'_{24}(p)$ le sous-graphe de $\mathrm{K}_{24}(p)$ obtenu en retirant le sommet correspondant au réseau de Leech. Alors le diamètre de $\mathrm{K}'_{24}(p)$ est inférieur ou  égal au diamètre de $\mathrm{K}_{23}(p)$.

En particulier que le graphe $\mathrm{K}'_{24}(p)$ est complet pour $p\geq 23$.
\end{cor}

Une dernière vérification de nos calculs provient de la proposition suivante :
\begin{prop} \label{adjoint} Soient $n=23$ ou $25$ et $p$ premier. Alors l'endomorphisme $\mathrm{T}_p$ de $\Z[X_n]$ est auto-adjoint pour le produit scalaire $(\ \vert \ )$ défini par :
$$(\overline{L}_1 \vert \overline{L}_2) = \vert \mathrm{O}(L_1) \vert \delta_{\overline{L}_1,\overline{L}_2},$$
où $\overline{L}_1,\overline{L}_2\in X_n$, $L_1\in \mathcal{L}_n$ a pour classe $\overline{L}_1$ dans $X_n$, $\mathrm{O}(L_1)$ est le groupe orthogonal de $L_1$, et $\delta_{\overline{L}_1,\overline{L}_2}=1$ si, et seulement si, $\overline{L}_1=\overline{L}_2$ et $0$ sinon.

Si l'on note $P_{23}$ (respectivement $P_{25}$) la matrice diagonale de $\mathrm{M}_{32}(\Z)$ (respectivement $\mathrm{M}_{121}(\Z)$) dont le $i$-ème terme diagonal est $\vert \mathrm{O}(L_i) \vert$ suivant la numérotation de la table \ref{rX23} (respectivement de la table \ref{rX25}), alors pour tout $p$ premier on a les égalités :
$$P_{23}\, \mathrm{T}_{23}(p) = \mathrm{T}^t_{23}(p)\, P_{23}\indent et \indent P_{25}\, \mathrm{T}_{25}(p) = \mathrm{T}^t_{25}(p)\, P_{25},$$
où pour $M\in \mathrm{M}_N(\Z)$ on désigne par $M^t$ la matrice transposée de $M$.
\end{prop}
\begin{proof}
Voir \cite{NV} et \cite[ch. III, \S 2]{CL}.
\end{proof}

D'après les valeurs obtenues pour $T_{23}$ et $T_{25}$, on vérifie qu'il existe qu'une seule matrice diagonale $P_1$ et une seule matrice diagonale $P_2$ (à multiplication près par un scalaire) telles que : $P_{1}\, T_{23}= T^t_{23}\, P_{1}$ et $P_{2}\, T_{25}= T^t_{25}\, P_{2}$.

L'existence de $P_1$ et $P_2$ constitue déjà en soi une vérification de nos résultats. On s'assure de plus que les coefficients de $P_1$ et $P_2$ sont cohérents avec les valeurs des $\vert \mathrm{O}(L_i) \vert$ que l'on sait déterminer facilement. C'est notamment le cas lorsque $R(L_i)$ et $L_i$ ont même rang (en tant que $\Z$-modules), et dans se cas le groupe $\mathrm{O}(L_i)$ se déduit directement du groupe de Weyl de $R(L_i)$.

Prenons par exemple le cas de $L_1$ (suivant les notations de la table \ref{rX23}). Posons $M=\mathrm{D}_{22}\oplus A_1$. On obtient alors : $\mathrm{r\acute{e}s}\, M = \left( \Z/2\Z\, d_1 \oplus^\perp \Z/2\Z\, d_2 \right) \oplus^\perp \Z/2\Z a$ avec $q(d_1)=q(d_2)=3/4$ et $q(a)=1/4$. En particulier, il existe deux droites isotropes dans $\mathrm{r\acute{e}s}\, M$, engendrées respectivement par $d_1+a$ et $d_2+a$. Ainsi, il existe deux réseaux $L^+,L^-\in \mathcal{L}_{23}$ contenant $M$ : ces deux réseaux sont isomorphes, et sont échangés par toute isométrie de $M$ qui n'est pas dans le groupe de Weyl de $M$ (qui sont exactement les isométries de $M$ qui échangent $d_1$ et $d_2$). Et on déduit que $\vert \mathrm{O}(L^+)\vert = \vert \mathrm{O}(L^-)\vert = \frac{\vert \mathrm{O}(M) \vert}{2}=22!\, 2^{22}$.

Il est aussi facile de déterminer la valeur de $\vert\mathrm{O}(L)\vert$ pour certains éléments de $\mathcal{L}_{25}$. En effet, la proposition \ref{adjoint} est aussi vraie pour $n=24$, suivant \cite[ch. III, \S 2]{CL} par exemple. On en déduit les valeurs des $\vert\mathrm{O}(P)\vert$ pour $P\in \mathcal{L}_{24}$ grâce à la matrice de $\mathrm{T}_2$ sur $\Z[X_{24}]$, donnée dans \cite{NV} par exemple. Si l'on se donne $P\in \mathcal{L}_{24}$, alors le réseau $L=P\oplus \mathrm{A}_1$ est un élément de $\mathcal{L}_{25}$, et on a l'égalité : $\vert \mathrm{O}(L)\vert = 2\, \vert \mathrm{O}(P)\vert$. Pour s'en convaincre, il suffit de constater qu'un élément $\gamma\in\mathrm{O}(L)$ satisfait $\gamma (P)=P$ et $\gamma (\mathrm{A}_1)=\mathrm{A}_1$. 

Ceci permet de déterminer la valeur de $\vert \mathrm{O}(L) \vert$ pour les éléments $L\in \mathcal{L}_{25}$ de la forme $L\simeq P\oplus \mathrm{A}_1$ avec $P\in \mathcal{L}_{24}$, qui sont exactement les éléments satisfaisant $R(L)\simeq R\coprod \mathbf{A}_1$, où $R$ est le système de racines d'un élément de $\mathcal{L}_{24}$. Les différentes valeurs de $R$ ont été déterminées par Niemeier \cite{Nie} et Venkov \cite{Ven} : c'est un système de racine de type ADE équicoxeter de rang $24$, ou l'ensemble vide (lorsque $P$ est le réseau de Leech).

Notons au passage que la détermination des matrices $P_{23}$ et $P_{25}$ nous donne pour tout $L$ dans $\mathcal{L}_{23}$ ou $\mathcal{L}_{25}$ la quantité $\vert \mathrm{O}(L) \vert$. On donne dans \cite{MegTp} les matrices $P_{23}$ et $P_{25}$.

\subsection{Congruences à la Harder.}\label{4.3}
On souhaite établir des congruences entre les traces de paramètres de Langlands-Satake de certaines formes automorphes normalisées, comme c'est par exemple le cas dans la conjecture de Harder, exposée dans  \cite{Har} et déjà démontrée dans \cite[Ch. X, Théorème 4.4]{CL} :

\begin{theo}[Conjecture de Harder] \label{harder} Pour tout premier $p$, on a la congruence suivante :
$$D_{21,5}(p)\equiv D_{21}(p)+p^{13}+p^8\mathrm{\ mod\ }41.$$
\end{theo}

Dans \cite{CL}, cette congruence est exprimée sous la forme $\tau_{4,10}(p)\equiv \tau_{22}(p)+p^{13}+p^8\mathrm{\ mod\ }41$, qui est équivalente, du fait des égalités $D_{21,5}(p)=\tau_{4,10}(p)$ et $D_{21}(p)=\tau_{22}(p)$ déjà expliquées au paragraphe \ref{2.3}.\\

De telles congruences entre les fonctions $D_{w_1,\dots,w_n}$ proviennent de congruences entre les colonnes des matrices $V$ ou $W$. On utilise pour cela le lemme suivant :

\begin{lem} \label{mlié} Soient $m$ un entier quelconque, et $i,j\in \{1,\dots,32\}$ tels que $\mathrm{vect}_{\Z/m\Z}(v_i)=\mathrm{vect}_{\Z/m\Z}(v_j)$. Alors pour tout $p$ on a la congruence :
$$\lambda_i(p) \equiv \lambda_j(p)\mathrm{\ mod\ }m.$$

De même, s'il existe $i,j\in \{1,\dots,57\}$ tels que $\mathrm{vect}_{\Z/m\Z}(w_i)=\mathrm{vect}_{\Z/m\Z}(w_j)$, alors pour tout $p$ on a la congruence :
$$\mu_i(p) \equiv \mu_j(p)\mathrm{\ mod\ }m.$$

Enfin, s'il existe $i,j\in \{58,\dots,121\}$ tels que $\mathrm{vect}_{\Z/m\Z}(w_i)=\mathrm{vect}_{\Z/m\Z}(w_j)$, alors pour tout $p$ on a la congruence :
$$\mu_i(p) \equiv \mu_j(p)\mathrm{\ mod\ }m\Z[\sqrt{144169}].$$
\end{lem}

\begin{proof}
Supposons par exemple qu'il existe un entier $m$, et des entiers $i,j\in \{1,\dots,32\}$ tels que $\mathrm{vect}_{\Z/m\Z}(v_i)=\mathrm{vect}_{\Z/m\Z}(v_j)$. Par définition, les coordonnées de $v_i$ sont premières entre elles (et il en va de même pour celles de $v_j$). En particulier, l'égalité précédente implique qu'il existe $\alpha \in (\Z/m\Z)^*$ tel que $v_i-\alpha v_j \in (m\Z)^{32}$. En faisant agir $T_{23}(p)\in \mathrm{M}_{32}(\Z)$ sur ce vecteur, on déduit les implications suivantes :
$$ \begin{aligned}
v_i-\alpha v_j\in (m\Z)^{32} &\Rightarrow T_{23}(p) (v_i-\alpha v_j)\in (m\Z)^{32} \\
&\Rightarrow \lambda_i(p) v_i -\alpha \lambda_j(p) v_j \in (m\Z)^{32} \\
&\Rightarrow \lambda_i(p) (v_i-\alpha v_j) + \alpha (\lambda_i(p) -\lambda_j(p))v_j \in (m\Z)^{32} \\
&\Rightarrow \alpha (\lambda_i(p) -\lambda_j(p))v_j \in (m\Z)^{32} \\
&\Rightarrow (\lambda_i(p) -\lambda_j(p))v_j \in (m\Z)^{32} \\
&\Rightarrow\lambda_i(p) -\lambda_j(p)\in m\Z .\\
\end{aligned}$$
(où la dernière implication utilise que les coordonnées de $v_j$ sont relativement premières entre elles).

La démonstration est identique si l'on possède des entiers $m\in \Z$ et $i,j\in \{1,\dots,121\}$ tels que $\mathrm{vect}_{\Z/m\Z}(w_i)=\mathrm{vect}_{\Z/m\Z}(w_j)$.

Dans le troisième cas, on peut remplacer l'hypothèse $\mathrm{vect}_{\Z/m\Z}(w_i)=\mathrm{vect}_{\Z/m\Z}(w_j)$ par $\mathrm{vect}_{\Z[\sqrt{144169}]/m\Z[\sqrt{144169}]}(w_i)=\mathrm{vect}_{\Z[\sqrt{144169}]/m\Z[\sqrt{144169}]}(w_j)$ et on obtient le même résultat. Cependant, on n'utilisera ce cas uniquement lorsque $\mu_i(p)-\mu_j(p)\in \Z$.
\end{proof}

Il suffit ensuite d'utiliser les proposition \ref{lambdai} et \ref{mui} pour déduire les congruences cherchées. Du fait des propositions \ref{lambdai} et \ref{mui}, les fonctions $D_{w_1,\dots,w_n}$ que l'on fera intervenir sont exactement celles provenant des représentations $\Delta_{w_1,\dots,w_n}$ présentes dans les tables \ref{psi23}, \ref{psi25} et \ref{psi252}. Par exemple, la proposition suivante améliore la congruence de Harder exposée précédemment :

\begin{theo} \label{harderopti} Pour tout nombre premier $p$, on a la congruence :
$$D_{21,5}(p)\equiv D_{21}(p)+p^{13}+p^8\mathrm{\ mod\ }9840.$$

De plus, cette congruence est optimale, dans le sens où on ne peut pas remplacer $9840$ par un de ses multiples.
\end{theo}

\begin{proof}
Pour obtenir la congruence précédente, il suffit dans un premier temps de constater que, pour $(m,i,j)=(9840,26,28)$, on a l'égalité : $\mathrm{vect}_{\Z/m\Z}(v_i)=\mathrm{vect}_{\Z/m\Z}(v_j)$. Le lemme \ref{mlié} nous donne pour tout premier $p$ la congruence $\lambda_{26}(p) \equiv \lambda_{28}(p)\mathrm{\ mod\ }9840$. D'après la table \ref{psi23} et la proposition \ref{lambdai}, on a les égalités :

$$
\begin{aligned}
\lambda_{26}(p) &= p^{\frac{21}{2}}\mathrm{Trace}\left( c_p\left( \Delta_{21} \oplus \Delta_{19,7} \oplus \Delta_{17} \oplus \Delta_{15} \oplus \Delta_{11} [3]\oplus [6] \right) \vert V_{\mathrm{St}}\right) \\
&= D_{21}(p)+p\, D_{19,7}(p) + p^2\, D_{17}(p)+p^3\, D_{15}(p)+p^4\, (1+p+p^2)\, D_{11}(p)\\
&\indent + p^8\, (1+p+p^2+p^3+p^4+p^5) ,\\
\lambda_{28}(p) & = p^{\frac{21}{2}}\mathrm{Trace}\left( c_p\left( \Delta_{21,5} \oplus \Delta_{19,7} \oplus \Delta_{17} \oplus \Delta_{15} \oplus \Delta_{11} [3]\oplus [4] \right) \vert V_{\mathrm{St}}\right) \\
&= D_{21,5}(p)+p\, D_{19,7}(p) + p^2\, D_{17}(p)+p^3\, D_{15}(p)+p^4\, (1+p+p^2)\, D_{11}(p)\\
&\indent + p^9\, (1+p+p^2+p^3) .\\
\end{aligned}
$$

Et ainsi :
$$\lambda_{26}(p)-\lambda_{28}(p) = D_{21}(p)+p^{13}+p^8-D_{21,5}(p)$$

qui est la congruence cherchée. Notons au passage que l'on a aussi l'égalité $\mathrm{vect}_{\Z/m\Z}(v_i)=\mathrm{vect}_{\Z/m\Z}(v_j)$ pour $(m,i,j)=(9840,27,29)$, qui nous donne la même congruence.

Pour constater que cette congruence est optimale, il suffit d'évaluer la quantité $\lambda_{26}(p)-\lambda_{28}(p)=D_{21}(p)+p^{13}+p^8-D_{21,5}(p)$ pour certaines valeurs de $p$. On constate par exemple que $\lambda_{26}(2)-\lambda_{28}(2)=9840$, ce qui prouve l'optimalité.
\end{proof}

L'ensemble des congruences que l'on a trouvées sont données par le théorème suivant :
\begin{theo} Pour tout nombre premier $p$, les congruences suivantes sont vérifiées :

\begin{enumerate}
\item[$(i)$] $D_{19,7}(p) \equiv D_{19}(p)+p^6+p^{13}\mathrm{\ mod\ }8712$ ;
\item[$(ii)$] $D_{21,5}(p) \equiv D_{21}(p)+p^8+p^{13}\mathrm{\ mod\ }9840$ ;
\item[$(iii)$] $D_{21,9}(p) \equiv (1+p^6)\,D_{15}(p)\mathrm{\ mod\ }12696$ ;
\item[$(iv)$] $D_{21,9}(p) \equiv D_{21}(p)+p^6+p^{15}\mathrm{\ mod\ }31200$ ;
\item[$(v)$] $D_{21,13}(p) \equiv (1+p^4)\, D_{17}(p)\mathrm{\ mod\ }8736$ ;
\item[$(vi)$] $D_{21,13}(p) \equiv D_{21}(p)+p^4+p^{17}\mathrm{\ mod\ }10920$ ;
\item[$(vii)$] $D_{23,7}(p)\equiv (1+p^8)\, D_{15}(p)\mathrm{\ mod\ }8972$ ;
\item[$(viii)$] $D_{23,13,5}(p)\equiv D_{23,13}(p)+p^9+p^{14}\mathrm{\ mod\ }5472$ ;
\item[$(ix)$] $D_{23,15,7}(p)\equiv (1+p^4)\, D_{19}(p)+p^8+p^{15}\mathrm{\ mod\ }2184$ ;
\item[$(x)$] $D_{23,15,7}(p)\equiv D_{23,7}(p)+p^4\, D_{15}(p)\mathrm{\ mod\ }5856$ ;
\item[$(xi)$] $D_{23,17,9}(p)\equiv D_{23,9}(p)+p^3\, D_{17}(p)\mathrm{\ mod\ }2976$ ;
\item[$(xii)$] $D_{23,19,3}(p)\equiv (1+p^2)\, D_{21}(p)+p^{10}+p^{13}\mathrm{\ mod\ }7872$ ;
\item[$(xiii)$] $D_{23,19,11}(p)\equiv (1+p^2)\, D_{21}(p)+p^6+p^{17}\mathrm{\ mod\ }16224$.
\end{enumerate} 

De plus, mis à part les points $(vi),(vii),(xi)$ et $(xiii)$, les congruences ci-dessus sont optimales, dans le sens où le module qui intervient ne peut pas être remplacé par un de ses multiples.
\end{theo}
\begin{proof}
On donne pour chacun des points de la proposition la congruence de la forme $\lambda_i\equiv \lambda_j\mathrm{\ mod\ }m$ (ou $\mu_i\equiv \mu_j\mathrm{\ mod\ }m$), qui se déduit du lemme \ref{mlié}. On précise lorsqu'il a fallu utiliser des résultats supplémentaires.

Certaines congruences qui apparaissent sont des ``multiplications par $p$" des congruences cherchées : on peut évaluer nos congruences jusqu'à $p=67$, et ainsi on pourra ``diviser une congruence modulo $m$ par $p$" dès lors que les diviseurs premiers de $m$ seront inférieurs ou égaux à $67$. C'est aussi le fait d'évaluer ces congruences jusqu'à $p=67$ qui nous permet dans certains cas de dire si elles sont optimales.

$(i)$ : $\lambda_{22}(p)\equiv \lambda_{26}(p)\mathrm{\ mod\ }17424$ ce qui nous donne la congruence $p\,D_{19,7}(p) \equiv p\,D_{19}(p)+p^7+p^{14}\mathrm{\ mod\ }17424$. Il suffit d'évaluer la quantité $D_{19,7}(p) \equiv D_{19}(p)+p^6+p^{13}$ pour les $p$ premiers divisant $17424$ pour déduire la congruence cherchée, qui est optimale.

$(ii)$ et $(iii)$ : $\lambda_{26}(p)\equiv \lambda_{28}(p)\mathrm{\ mod\ }9840$ et $\lambda_{16}(p)\equiv \lambda_{19}(p)\mathrm{\ mod\ }12696$ donnent directement les congruences cherchées, qui sont optimales.

$(iv)$ : $\lambda_{20}(p)\equiv \lambda_{14}(p)\mathrm{\ mod\ }7800$ et $\mu_{15}(p)\equiv \mu_{24}(p)\mathrm{\ mod\ }20800$ nous donnent la congruence $p\, D_{21,9}(p) \equiv p\, D_{21}(p)+p^7+p^{16}\mathrm{\ mod\ }62400$. Il suffit d'évaluer la quantité $D_{21,9}(p) \equiv D_{21}(p)+p^6+p^{15}$ pour les $p$ premiers divisant $62400$ pour déduire la congruence cherchée, qui est optimale.

$(v)$ : $\lambda_{11}(p)\equiv \lambda_{12}(p)\mathrm{\ mod\ }4368$ et $\lambda_{17}(p)\equiv \lambda_{18}(p)\mathrm{\ mod\ }416$ donnent directement la congruence cherchée, qui est optimale.

$(vi)$ : $\lambda_{10}(p)\equiv \lambda_{7}(p)\mathrm{\ mod\ }2730$ et $\mu_{64}(p)\equiv \mu_{67}(p)\mathrm{\ mod\ }8$ nous donnent la congruence $p\,D_{21,13}(p) \equiv p\,D_{21}(p)+p^5+p^{18}\mathrm{\ mod\ }10920$ et il suffit d'évaluer la quantité $D_{21,13}(p) \equiv D_{21}(p)+p^4+p^{17}$ pour les $p$ premiers divisant $10920$ pour obtenir la congruence cherchée. Cependant, on n'a aucune certitude sur l'optimalité de la congruence.

$(vii)$ à $(xii)$ : $\mu_{10}\equiv \mu_{18}\mathrm{\ mod\ }8972$, $\mu_{44}\equiv \mu_{37}\mathrm{\ mod\ }5472$, $\mu_{21}\equiv \mu_{36}\mathrm{\ mod\ }2184$, $\mu_{41}\equiv \mu_{38}\mathrm{\ mod\ }5856$, $\mu_{35}\equiv \mu_{40}\mathrm{\ mod\ }2976$ et $\mu_{53}\equiv \mu_{56}\mathrm{\ mod\ }7872$ nous donnent directement les congruences cherchées. Il est facile de voir que les congruences $(viii)$, $(ix)$, $(x)$ et $(xii)$ sont optimales.

$(xiii)$ : les congruences $\mu_{11}\equiv \mu_{6}\mathrm{\ mod\ }5408$, $\mu_{11}\equiv \mu_{12}\mathrm{\ mod\ }96$, $\mu_{12}\equiv \mu_{6}\mathrm{\ mod\ }44224$ et $\mu_{13}\equiv \mu_{7}\mathrm{\ mod\ }8292$ nous donnent les congruences : $D_{23,19,11}(p)\equiv (1+p^2)\, D_{21}(p)+p^6+p^{17}\mathrm{\ mod\ }5408$, $D_{23,19,11}(p)\equiv (1+p^2)\, D_{21}(p)+p^6\,D_{11}(p)\mathrm{\ mod\ }96$ et $p^6\, D_{11}(p)\equiv p^6+p^{17}\mathrm{\ mod\ }132672$ (qui est une ``multiplication par $p^6$" de la congruence de Ramanujan). On en déduit ainsi la congruence cherchée.
\end{proof}

Outre la congruence $(ii)$ (qui est une optimisation de la conjecture de Harder), deux congruences attirent notre attention.

La congruence $(vii)$ est la seule faisant intervenir une forme de Siegel de genre $2$ qui n'avait pas été traitée dans \cite{CL} ou \cite{Tay}. Suivant la méthode de \cite{Die} développée dans \cite{Tay}, on a le résultat suivant :
\begin{prop} Soit $S$ l'espace des formes modulaires paraboliques de Siegel de genre $2$ de poids $\mathrm{Sym}^6\C^2\otimes \mathrm{det}^{10}$ (qui est un espace de dimension $1$). Pour $l$ un nombre premier, on note $\rho_l$ la représentation $l$-adique associée par la construction de Weissauer (voir \cite{Wei} et \cite[Théorème 1.1]{Tay}), et $\overline{\rho}_l$ la $\F_l$-représentation résiduelle associée (voir par exemple \cite[\S X.1]{CL}).

Alors pour $l>23$, la représentation $\overline{\rho}_l$ est irréductible si, et seulement si, $l\neq 2243$.
\end{prop}
\begin{proof}
Il suffit simplement d'évaluer les quantités $A_{6,10}(p)$, $B_{6,10}(p)$, $C_{6,10}(p)$ et $D_{6,10}(p)$ (suivant les notations de \cite{Tay}) pour certaines valeurs de $p$, et de regarder les diviseurs supérieurs à $23$ du $\mathrm{pgcd}$ de ces quantités. On trouve les résultats suivants :
$$\begin{array}{rcl}
\mathrm{pgcd}\left( \left\{ A_{6,10}(p) \, \vert\, p\in \{2,3,5\} \right\} \right) & = & 1 \\
\mathrm{pgcd}\left( \left\{ B_{6,10}(p) \, \vert\, p\in \{2,3,5\} \right\} \right) & = & 2^9\cdot 3^3\cdot 5\cdot 11\cdot 13^2\\
\mathrm{pgcd}\left( \left\{ C_{6,10}(p) \, \vert\, p\in \{2,3,5\} \right\} \right) & = & 2^{11}\cdot 3^2\cdot 2243\\
\mathrm{pgcd}\left( \left\{ D_{6,10}(p) \, \vert\, p\in \{2,3,5\} \right\} \right) & = & 2^8\cdot 3^2\cdot 5^2\\
\end{array}$$
\noindent ce qui donne bien le résultat voulu.

Comme il y a seulement les quantités $C_{6,10}(p)$ qui font intervenir le nombre premier $l=2243$, on en déduit que la seule réductibilité de $\overline{\rho}_{l}$ est du type $\pi_1\oplus \pi_1^*\otimes \chi_l^{23}$, où $\pi_1$ est une représentation galoisienne irréductible de dimension $2$ provenant d'une forme modulaire (suivant la construction classique de Serre \cite{Ser} et Deligne \cite{Del}) dont les poids de l'ensemble des poids de l'inertie modérée est $\{0,15\}$. La représentation $\pi_1$ est nécessairement celle associée à l'unique forme modulaire parabolique normalisée de poids $16$.

Ainsi, la congruence $(vii)$ est la seule congruence possible provenant de la réductibilité des représentations $\overline{\rho}_l$, telle que le module $m$ possède un diviseur premier $l>23$.
\end{proof}

La seconde congruence qui attire notre attention est la congruence $(viii)$. En effet, elle apparaissait déjà dans \cite[\S 6, Example 3]{BDM} sous la forme plus faible de la conjecture :
$$\left( \forall \, p\text{ premier}\right)\  D_{23,13,5}(p)\equiv D_{23,13}(p)+p^9+p^{14}\ \mathrm{mod}\ 19.$$
Ainsi, en constatant que $5472=2^5\cdot 3^2\cdot 19$, la congruence $(viii)$ apparaît donc non seulement comme un démonstration de la conjecture précédente, mais aussi comme une amélioration comme on a pu remplacer $19$ par un de ses multiples.

\subsection{Une conjecture de Gan-Gross-Prasad.}\label{4.4}
\indent Soit $n\equiv -1,0\mathrm{\ mod\ }8$. Alors on possède une application naturelle $\phi_n:X_n\rightarrow X_{n+1}$ définie de la manière suivante :

- si $n\equiv -1\mathrm{\ mod\ }8$ : on se donne $L\in \mathcal{L}_n$, et on note $\overline{L}$ sa classe dans $X_n$. Alors $M=L\oplus A_1$ est un réseau de $\R^{n+1}$ tel que $\mathrm{r\acute{e}s}M =\mathrm{r\acute{e}s}L \oplus \mathrm{r\acute{e}s} A_1\simeq \Z/2\, e_1 \oplus \Z/2\, e_2$ (avec $q(e_1)=3/4$ et $q(e_2)=1/4$). En particulier, $\mathrm{r\acute{e}s}M$ possède une unique droite isotrope (celle engendrée par $e_1+e_2$) : l'image réciproque de cette droite par la projection naturelle $M^\sharp \rightarrow \mathrm{r\acute{e}s}M$ est un réseau $P\in \mathcal{L}_{n+1}$ dont la classe $\overline{P}$ dans $X_{n+1}$ ne dépend que de $\overline{L}$. On pose $\phi_n(\overline{L}) = \overline{P}$.

- si $n\equiv 0\mathrm{\ mod\ }8$ : on se donne $L\in \mathcal{L}_n$, et on note $\overline{L}$ sa classe dans $X_n$. Alors $P=L\oplus A_1$ est un élément de $\mathcal{L}_{n+1}$ dont la classe $\overline{P}$ dans $X_{n+1}$ ne dépend que de $\overline{L}$. On pose $\phi_n(\overline{L}) = \overline{P}$.\\

On sait déjà que, pour $n=23,24$ ou $25$, deux éléments de $\mathcal{L}_n$ sont isomorphes si, et seulement si, leurs systèmes de racines sont isomorphes. Les application $\phi_{23}$ et $\phi_{24}$ se déduisent directement de cette constatation :

- pour $\phi_{23}$ : soient $L\in \mathcal{L}_{23}$, $P\in \mathcal{L}_{24}$ et $\overline{L},\overline{P}$ leurs classes respectives dans $X_{23}$ et $X_{24}$. Alors $\phi_{23}(\overline{L})=\overline{P}$ si, et seulement si, il existe $\alpha\in R(P)$ tel que $L\simeq P\cap \alpha^\perp$. L'application $\phi_{23}$ se déduit donc directement de la table \ref{RX23}.

- pour $\phi_{25}$ : soient $L\in \mathcal{L}_{24}$, $P\in \mathcal{L}_{25}$ et $\overline{L},\overline{P}$ leurs classes respectives dans $X_{24}$ et $X_{25}$. Alors $\phi_{24}(\overline{L})=\overline{P}$ si, et seulement si, $R(P)\simeq R(L)\oplus \mathbf{A}_1$.\\

Les applications $\phi_n: X_n\rightarrow X_{n+1}$ ainsi définies induisent des applications $\C [X_n] \rightarrow \C[X_{n+1}]$. Du fait des isomorphismes $\C^{X_{n}}\simeq \C[X_n]^*$ et $\C^{X_{n+1}}\simeq \C[X_{n+1}]^*$, on déduit que $\phi_n$ définit une application : $\widetilde{\phi}_n : \C^{X_{n+1}} \rightarrow \C^{X_{n}}$. Pour simplifier les notations, si $f\in \C^{X_{n+1}}$, on note $f\vert_{X_n}=\widetilde{\phi}_n (f)$.
\begin{defi} Soient $n\equiv -1\mathrm{\ mod\ }8$ (respectivement $n\equiv 0\mathrm{\ mod\ }8$), et $\pi\in \Pi_{\mathrm{disc}}(\mathrm{O}_{n+1})$ (respectivement $\pi\in \Pi_{\mathrm{cusp}}(\mathrm{SO}_{n+1})$) tel que $\pi_{\infty}=\C$. On lui associe le sous-ensemble $\mathrm{Res}\, \pi$ de $\Pi_{\mathrm{cusp}}(\mathrm{SO}_n)$ (respectivement $\Pi_{\mathrm{disc}}(\mathrm{O}_{n})$) défini comme suit.

Soit $f\in \mathcal{M}_\C (\mathrm{O}_{n+1})$ qui engendre $\pi$ (en particulier, $f$ est propre pour $\mathrm{H}(\mathrm{O}_{n+1})$). On écrit $f\vert_{X_n}=f_1+\dots+f_r$, où les $f_i$ sont propres pour $\mathrm{H}(\mathrm{O}_n)$. Alors $\pi'\in \mathrm{Res}\, \pi$ si, et seulement si, l'un des $f_i$ ci-dessus engendre $\pi'$.
\end{defi}
\begin{prop} Les tables \ref{res23}, \ref{res24} et \ref{res242} donnent les paramètres standards des éléments de $\mathrm{Res}\, \pi$ pour $\pi\in \Pi_{\mathrm{disc}}(\mathrm{O}_{24})$ ou $\pi \in \Pi_{\mathrm{cusp}}(\mathrm{SO}_{25})$, lorsque $\pi_\infty$ est trivial.
\end{prop}
\begin{proof}
On se contente ici de détailler le cas où $\pi \in \Pi_{\mathrm{disc}}(\mathrm{O}_{24})$ (la même méthode s'applique lorsque $\pi\in\Pi_{\mathrm{cusp}}(\mathrm{SO}_{25})$).

On possède une base $v_i$ de vecteurs propres pour l'action de $\mathrm{H}(\mathrm{O}_{23})$ sur $\C[X_{23}]$, ainsi qu'une base $w_j$ de vecteurs propres pour l'action de $\mathrm{H}(\mathrm{O}_{24})$ sur $\C[X_{24}]$ (grâce à la matrice $T_{23}$ déterminée précédemment, et à la matrice de $\mathrm{T}_2$ sur $\Z[X_{24}]$ donnée dans \cite{NV}). On connaît de plus les représentations $\pi'_i\in \Pi_{\mathrm{cusp}}(\mathrm{SO}_{23})$ et $\pi_j\in \Pi_{\mathrm{disc}}(\mathrm{O}_{24})$ respectivement associées aux $v_i$ et aux $w_j$ (voir table \ref{psi23} et \cite[Table C.5]{CL} par exemple).

Comme on a déterminé l'application $\phi_{23}:\C[X_{23}]\rightarrow \C[X_{24}]$, on peut construire les éléments $\alpha_{i,j}\in \C$ vérifiant : $(\forall \, i\in \{1,\dots,32\})\ \phi_{23}(v_i) = \sum_{j=1}^{24} \alpha_{i,j} w_j.$

Par définition, l'ensemble $\mathrm{Res}\, \pi_j$ se déduit de l'équivalence : $\pi'_i\in \mathrm{Res}\, \pi_j \Leftrightarrow \alpha_{i,j}\neq 0.$
\end{proof}
\begin{theo} Soient $n=1,2$ ou $3$, $\pi\in \Pi_{\mathrm{disc}}(\mathrm{O}_{8n})$ (respectivement $\pi\in \Pi_{\mathrm{cusp}}(\mathrm{SO}_{8n+1})$) tel que $\pi_\infty$ est triviale, et $\pi'\in  \Pi_{\mathrm{cusp}}(\mathrm{SO}_{8n-1})$ (respectivement $\pi'\in  \Pi_{\mathrm{disc}}(\mathrm{O}_{8n})$) tel que $\pi'_\infty$ est aussi triviale.

Alors $\pi'\in \mathrm{Res}\, \pi$ si, et seulement si, il existe des multi-ensembles finis $\Pi_1,\Pi_2,\Pi_3\subset \cup_{m}\Pi_{\mathrm{cusp}}(\mathrm{PGL}_m)$, et des applications $d_1:\Pi_1\rightarrow \N^*$ et $d_2:\Pi_2\rightarrow \N^*$ tels que :
$$\psi (\pi,\mathrm{St}) = \bigoplus_{\pi_i\in \Pi_1} \pi_i [d_1( \pi_i)] \ \oplus\ \bigoplus_{\pi_j\in \Pi_2} \pi_j [d_2( \pi_j)],$$
$$\psi (\pi',\mathrm{St}) = \bigoplus_{\pi_i\in \Pi_1} \pi_i [d_1( \pi_i)+1] \ \oplus\ \bigoplus_{\pi_j\in \Pi_2} \pi_j [d_2( \pi_j)-1] \ \oplus\ \bigoplus_{\pi_k\in \Pi_3} \pi_k.$$
\end{theo}
\begin{proof}
Le cas $n=3$ découle directement d'une inspection des tables \ref{res23}, \ref{res24} et \ref{res242}. Le cas $n=2$ se déduit facilement des matrices de l'opérateur $\mathrm{T}_2$ sur $\C[X_{15}]$, $\C[X_{16}]$ et $\C[X_{17}]$ (données dans \cite[Annexe B]{CL} pour $X_{15}$ et $X_{17}$, et dans \cite[ch. III, \S 3]{CL} pour $X_{16}$ par exemple). Le cas $n=1$ est évident (comme $\vert X_7\vert = \vert X_8\vert = \vert X_9\vert =1$).
\end{proof}

Ce résultat valide dans ces cas particuliers la conjecture formulée par Gan-Gross-Prasad, qui conclut l'exposé \cite[Classical groups, the local case]{GGP}. Suivant les mêmes notations, il s'agit des cas où $n=m-1$, avec $m\in \{8,9,16,17,24,25\}$.
\clearpage
\section{Tables de résultats.}\label{5}
\begin{table}[h!] \begin{bigcenter} \renewcommand\arraystretch{1.2} 
\begin{tabular}{|c|c|c||c|c|c|}
\hline 
\rule[-0.3cm]{0mm}{0.8cm} $R(P)$ & $R$ & $R(L)$ & $R(P)$ & $R$ & $R(L)$ \\
\hline
$\mathbf{D}_{24}$ & $\mathbf{D}_{24}$ & $\mathbf{D}_{22}\coprod \mathbf{A}_1$ & $\mathbf{A}_{11}\coprod \mathbf{D}_7\coprod \mathbf{E}_6$ & $\mathbf{E}_6$ & $\mathbf{A}_{11}\coprod \mathbf{D}_7\coprod \mathbf{A}_5$ \\
\hline
\multirow{2}{*}{$\mathbf{D}_{16}\coprod \mathbf{E}_8$} & $\mathbf{D}_{16}$ & $\mathbf{D}_{14}\coprod \mathbf{E}_8\coprod \mathbf{A}_1 $ &  $4\, \mathbf{E}_6$ & $\mathbf{E}_6$ & $3\, \mathbf{E}_6 \coprod \mathbf{A}_5$\\
\cline{4-6}
 & $\mathbf{E}_{8}$ & $\mathbf{D}_{16}\coprod \mathbf{E}_7$ & \multirow{2}{*}{$2\, \mathbf{A}_9 \coprod \mathbf{D}_6$} & $\mathbf{A}_{9}$ & $\mathbf{A}_{9}\coprod \mathbf{A}_7\coprod \mathbf{D}_6 $ \\
\cline{1-3}
$3\, \mathbf{E}_{8}$ & $\mathbf{E}_{8}$ & $2\, \mathbf{E}_{8}\coprod \mathbf{E}_7$ & & $\mathbf{D}_6$ & $2\, \mathbf{A}_9 \coprod \mathbf{D}_4\coprod \mathbf{A}_1$ \\
\hline
$\mathbf{A}_{24}$ & $\mathbf{A}_{24}$ & $\mathbf{A}_{22}$ & $4\,\mathbf{D}_6$ & $\mathbf{D}_6$ & $3\, \mathbf{D}_6 \coprod \mathbf{D}_4 \coprod \mathbf{A}_1$\\
\hline
$2\, \mathbf{D}_{12}$ & $\mathbf{D}_{12}$ & $\mathbf{D}_{12}\coprod \mathbf{D}_{10}\coprod \mathbf{A}_1$ & $3\, \mathbf{A}_8$ & $\mathbf{A}_8$ & $2\, \mathbf{A}_8 \coprod \mathbf{A}_6$ \\
\hline
\multirow{2}{*}{$\mathbf{A}_{17}\coprod \mathbf{E}_7$} & $\mathbf{A}_{17}$ & $\mathbf{A}_{15}\coprod \mathbf{E}_7$ & \multirow{2}{*}{$2\, \mathbf{A}_{7}\coprod 2\, \mathbf{D}_5$} & $\mathbf{A}_{7}$ & $\mathbf{A}_{7}\coprod 2\, \mathbf{D}_5\coprod \mathbf{A}_5 $ \\
\hline
\end{tabular} \caption{Classes d'isomorphisme des systèmes de racines des éléments de $X_{23}$}\label{RX23} \end{bigcenter} \end{table}
\begin{table}[h!] \begin{bigcenter} \renewcommand\arraystretch{1.2} \begin{tabular}{|c|c|c||c|c|c|}
\hline 
\rule[-0.3cm]{0mm}{0.8cm} $i$ & $R_i$ & $\phi_i$ & $i$ & $R_i$ & $\phi_i$ \\
\hline
\small $1$ & \small $\mathbf{D}_{22} \coprod \mathbf{A}_1$ & \small $(926,1,3,18,1,8)$ & \small $17$ & \small $\mathbf{A}_{11} \coprod \mathbf{D}_7 \coprod \mathbf{A}_5$ & \small $(246,0,7,15,1,288)$\\ 
\hline 
\small $2$ & \small $\mathbf{D}_{14} \coprod \mathbf{E}_8 \coprod \mathbf{A}_1$ & \small $(606,1,6,14,2,8)$&  $18$ & \small $3\,\mathbf{E}_6 \coprod \mathbf{A}_5$ & \small $(246,0,11,9,3,162)$ \\ 
\hline 
\small $3$ & \small $\mathbf{D}_{16} \coprod \mathbf{E}_7$ & \small $(606,0,6,15,2,8)$ & \small $19$ & \small $\mathbf{A}_9 \coprod \mathbf{A}_7 \coprod \mathbf{D}_6$ & \small $(206,0,7,14,1,320)$\\ 
\hline 
\small $4$ & \small $2\,\mathbf{E}_8 \coprod \mathbf{E}_7$ & \small $(606,0,9,11,3,2)$& \small $20$ & \small $2\,\mathbf{A}_9 \coprod \mathbf{D}_4 \coprod \mathbf{A}_1$ & \small $(206,1,7,14,1,800)$\\ 
\hline 
\small $5$ & \small $\mathbf{A}_{22} $ & \small $(506,0,2,20,0,23)$ & \small $21$ & \small $3\,\mathbf{D}_6 \coprod \mathbf{D}_4 \coprod \mathbf{A}_1$ & \small $(206,1,12,6,4,512)$\\ 
\hline 
\small $6$ & \small $\mathbf{D}_{12} \coprod \mathbf{D}_{10} \coprod \mathbf{A}_1$ & \small $(446,1,6,14,2,32)$ & \small $22$ & \small $2\,\mathbf{A}_8 \coprod \mathbf{A}_6$ & \small $(186,0,6,16,0,567)$\\ 
\hline 
\small $7$ & \small $\mathbf{A}_{15} \coprod \mathbf{E}_7$ & \small $(366,0,5,16,1,32)$ & \small $23$ & \small $\mathbf{A}_7 \coprod 2\,\mathbf{D}_5 \coprod \mathbf{A}_5$ & \small $(166,0,10,10,2,768)$ \\
\hline
\end{tabular} \caption{Numérotations des systèmes de racines des éléments de $X_{23}$}\label{rX23} \end{bigcenter} \end{table}
\begin{table}[h!] \begin{bigcenter} \renewcommand\arraystretch{0.7} \begin{tabular}{|m{0.17cm}|m{4cm}|m{3.2cm}||m{0.28cm}|m{3.6cm}|m{2.9cm}|}
\hline 
\centering $i$ & \rule[-0.2cm]{0mm}{0.5cm} \centering $R_i$ & \centering $\phi_i$ & \centering $i$ & \centering $R_i$ & \centering $\phi_i$ \tabularnewline
\hline
\footnotesize \centering $1$ &\footnotesize \centering $\mathbf{A}_{1}$ &\footnotesize \centering $(2, 1, 0, 0, 0, 2)$ &\footnotesize \centering $62$ &\footnotesize \centering $\mathbf{A}_{8} \coprod 2\,\mathbf{A}_{7} \coprod \mathbf{A}_{1}$ &\footnotesize \centering $(186, 1, 6, 16, 0, 1152)$ \tabularnewline 
\hline
\footnotesize \centering $2$ &\footnotesize \centering $\mathbf{A}_{2}$ &\footnotesize \centering $(6, 0, 2, 0, 0, 3)$ &\footnotesize \centering $63$ &\footnotesize \centering $\mathbf{A}_{9} \coprod \mathbf{D}_{5} \coprod \mathbf{A}_{5} \coprod \mathbf{D}_{4} \coprod \mathbf{A}_{1}$ &\footnotesize \centering $(186, 1, 10, 11, 2, 1920)$ \tabularnewline 
\hline
\footnotesize \centering $3$ &\footnotesize \centering $9\,\mathbf{A}_{1}$ &\footnotesize \centering $(18, 9, 0, 0, 0, 512)$ &\footnotesize \centering $64$ &\footnotesize \centering $\mathbf{A}_{9} \coprod \mathbf{A}_{7} \coprod 2\,\mathbf{A}_{4}$ &\footnotesize \centering $(186, 0, 8, 16, 0, 2000)$ \tabularnewline 
\hline
\footnotesize \centering $4$ &\footnotesize \centering $\mathbf{A}_{2} \coprod 12\,\mathbf{A}_{1}$ &\footnotesize \centering $(30, 12, 2, 0, 0, 12288)$ &\footnotesize \centering $65$ &\footnotesize \centering $\mathbf{E}_6 \coprod 3\,\mathbf{A}_{6}$ &\footnotesize \centering $(198, 0, 9, 14, 1, 1029)$ \tabularnewline 
\hline
\footnotesize \centering $5$ &\footnotesize \centering $25\,\mathbf{A}_{1}$ &\footnotesize \centering $(50, 25, 0, 0, 0, 33554432)$ &\footnotesize \centering $66$ &\footnotesize \centering $\mathbf{A}_{9} \coprod \mathbf{A}_{8} \coprod \mathbf{A}_{5} \coprod \mathbf{A}_{2}$ &\footnotesize \centering $(198, 0, 8, 16, 0, 1620)$ \tabularnewline
\hline
\end{tabular} \caption{Numérotations des systèmes de racines des éléments de $X_{25}$}\label{rX25} \end{bigcenter} \end{table}
\begin{table}[h!] \begin{bigcenter} \renewcommand\arraystretch{1.2} \begin{tabular}{|c|c|c|}
\hline 
\rule[-0.3cm]{0mm}{0.8cm} $i$ & $\psi_i$ & $\lambda_i$\\
\hline
$1$ & $[22]$ & $4194303$ \\
\hline
$2$ & $\Delta_{21} \oplus [20]$ & $2096862$ \\
\hline
$3$ & $\Delta_{21} \oplus \Delta_{19} \oplus [18]$ & $1049196$ \\
\hline
$4$ & $\Delta_{19} [3]\oplus [16]$ & $527472$ \\
\hline
$5$ & $\Delta_{21} \oplus \Delta_{19} \oplus \Delta_{17} \oplus [16]$ & $522792$ \\
\hline
$6$ & $\Delta_{19} [3]\oplus \Delta_{15} \oplus [14]$ & $267048$ \\
\hline
$7$ & $\Delta_{21} \oplus \Delta_{19} \oplus \Delta_{17} \oplus \Delta_{15} \oplus [14]$ & $262368$ \\
\hline
\end{tabular} \caption{Paramètres standards $\psi_i=\psi (\pi_i,\mathrm{St})$ des $32$ représentations $\pi_i$ dans $\Pi_{\mathrm{cusp}}(\mathrm{SO}_{23})$ telles que $\pi_{\infty}$ est triviale, rangées par valeurs propres associées pour $\mathrm{T}_2$ décroissantes.}\label{psi23} \end{bigcenter} \end{table}
\begin{table}[h!] \begin{bigcenter} \renewcommand\arraystretch{0.75} \begin{tabular}{|c|c|c|}
\hline 
\rule[-0.3cm]{0mm}{0.8cm} $i$ & $\psi_i$ & $\mu_i$\\
\hline
\small $ 1 $ & \small $[24]$ & \small $ 16777215 $ \\
\hline
\small $ 2 $ & \small $\Delta_{21} [3]\oplus [18]$ & \small $ 2095128 $ \\
\hline
\small $ 3 $ & \small $\Delta_{21} [3]\oplus \Delta_{17}\oplus [16]$ & \small $ 1042320 $ \\
\hline
\small $ 4 $ & \small $\Delta_{19} [5]\oplus [14]$ & \small $ 538392 $ \\
\hline
\small $ 5 $ & \small $\Delta_{21} [3]\oplus \Delta_{17}\oplus \Delta_{15}\oplus [14]$ & \small $ 521472 $ \\
\hline
\small $ 6 $ & \small $\Delta_{21} [3]\oplus \Delta_{15} [3]\oplus [12]$ & \small $ 272160 $ \\
\hline
\small $ 7 $ & \small $\Delta_{23,13}\oplus \Delta_{19} [3]\oplus \Delta_{15}\oplus [12]$ & \small $ 271440 $ \\
\hline
\end{tabular} \caption{Paramètres standards $\psi_i=\psi (\pi_i,\mathrm{St})$ des $57$ représentations $\pi_i$ dans $\Pi_{\mathrm{cusp}}(\mathrm{SO}_{25})$ telles que $\pi_{\infty}$ est triviale pour lesquelles les valeurs propres associées pour $\mathrm{T}_2$ sont entières, rangées par valeurs propres associées pour $\mathrm{T}_2$ décroissantes.}\label{psi25} \end{bigcenter} \end{table}
\begin{table}[h!] \begin{bigcenter} \renewcommand\arraystretch{1.2} \begin{tabular}{|c|c|c|}
\hline 
\rule[-0.3cm]{0mm}{0.8cm} $i$ & $\psi_i,\ \psi_{i+1}$ & $\mu_i,\ \mu_{i+1}$\\
\hline
$ 58 $ & $\Delta_{23}^2 \oplus [22]$ & $ 8389146 \pm 12\sqrt{144169} $ \\
\hline
$ 60 $ & $\Delta_{23}^2 \oplus \Delta_{21} \oplus [20]$ & $ 4194264 \pm 12\sqrt{144169} $ \\
\hline
$ 62 $ & $\Delta_{23}^2 \oplus \Delta_{21} \oplus \Delta_{19} \oplus [18]$ & $ 2098932 \pm 12\sqrt{144169} $ \\
\hline
$ 64 $ & $\Delta_{23}^2 \oplus \Delta_{19} [3] \oplus [16]$ & $ 1055484 \pm 12\sqrt{144169} $ \\
\hline
$ 66 $ & $\Delta_{23}^2 \oplus \Delta_{21} \oplus \Delta_{19} \oplus \Delta_{17} \oplus [16]$ & $ 1046124 \pm 12\sqrt{144169} $ \\
\hline
$ 68 $ & $\Delta_{23}^2 \oplus \Delta_{19} [3] \oplus \Delta_{15} \oplus [14]$ & $ 534636 \pm 12\sqrt{144169} $ \\
\hline
$ 70 $ & $\Delta_{23}^2 \oplus \Delta_{21} \oplus \Delta_{19} \oplus \Delta_{17} \oplus \Delta_{15} \oplus [14]$ & $ 525276 \pm 12\sqrt{144169} $ \\
\hline
\end{tabular} \caption{Paramètres standards $\psi_i=\psi (\pi_i,\mathrm{St})$ des $64$ représentations $\pi_i$ dans $\Pi_{\mathrm{cusp}}(\mathrm{SO}_{25})$ telles que $\pi_{\infty}$ est triviale pour lesquelles les valeurs propres associées pour $\mathrm{T}_2$ ne sont pas entières, rangées par valeurs propres associées pour $\mathrm{T}_2$ décroissantes.}\label{psi252} \end{bigcenter} \end{table}
\begin{table}[h!] \begin{bigcenter} \renewcommand\arraystretch{1} \begin{tabular}{|c|m{9cm}|}
\hline 
\rule[-0.3cm]{0mm}{0.8cm} $\psi (\pi,\mathrm{St})$ & \centering $\psi (\pi',\mathrm{St})$ \tabularnewline
\hline
$[23]\oplus [1]$ & \footnotesize \centering$[22]$ \tabularnewline
\hline
$\mathrm{Sym}^2\Delta_{11}\oplus [21]$ & \footnotesize \centering$[22],\ \Delta_{21}\oplus [20]$ \tabularnewline
\hline
$\Delta_{21}[2]\oplus [1]\oplus [19]$ & \footnotesize \centering$\Delta_{21} \oplus [20],\ \Delta_{21} \oplus \Delta_{19} \oplus [18]$\tabularnewline
\hline
$\mathrm{Sym}^2\Delta_{11}\oplus \Delta_{19}[2]\oplus [17]$ & \footnotesize \centering$\Delta_{21} \oplus \Delta_{19} \oplus [18],\  \Delta_{19} [3]\oplus [16],\  \Delta_{21} \oplus \Delta_{19} \oplus \Delta_{17} \oplus [16]$\tabularnewline
\hline
$\Delta_{19}[4]\oplus [1]\oplus [15]$ & \footnotesize \centering$\Delta_{19} [3]\oplus [16],\  \Delta_{19} [3]\oplus \Delta_{15} \oplus [14]$\tabularnewline
\hline
$\Delta_{21}[2]\oplus \Delta_{17}[2]\oplus [1]\oplus [15]$ & \footnotesize \centering$\Delta_{21} \oplus \Delta_{19} \oplus \Delta_{17} \oplus [16],\  \Delta_{21} \oplus \Delta_{19} \oplus \Delta_{17} \oplus \Delta_{15} \oplus [14],$ \newline $ \Delta_{21} \oplus \Delta_{17} [3]\oplus [14]$\tabularnewline
\hline
$\mathrm{Sym}^2\Delta_{11} \oplus \Delta_{19}[2]\oplus \Delta_{15}[2]\oplus [13]$ & \footnotesize \centering$\Delta_{19} [3]\oplus \Delta_{15} \oplus [14],\  \Delta_{21} \oplus \Delta_{19} \oplus \Delta_{17} \oplus \Delta_{15} \oplus [14],$ \newline $ \Delta_{21} \oplus \Delta_{19} \oplus \Delta_{15} [3]\oplus [12],\  \Delta_{21,13} \oplus \Delta_{19} \oplus \Delta_{17} \oplus \Delta_{15} \oplus [12]$\tabularnewline
\hline 
\end{tabular} \caption{Paramètres standards $\psi(\pi',\mathrm{St})$ des éléments $\pi'\in \mathrm{Res}\, \pi$, pour les représentations $\pi\in \Pi_{\mathrm{disc}}(\mathrm{O}_{24})$ telles que $\pi_{\infty}$ est triviale.}\label{res23} \end{bigcenter} \end{table}
\begin{table}[h!] \begin{bigcenter} \renewcommand\arraystretch{0.65} \begin{tabular}{|c|c|}
\hline 
\rule[-0.3cm]{0mm}{0.8cm} $\psi (\pi,\mathrm{St})$ & $\psi (\pi',\mathrm{St})$ \\
\hline
\small $[24]$ & \small $[23]\oplus [1]$ \\
\hline
\small $\Delta_{21} [3]\oplus [18]$ & \small $\Delta_{21}[2]\oplus [1]\oplus [19]$ \\
\hline
\small $\Delta_{21} [3]\oplus \Delta_{17}\oplus [16]$ & \small $\Delta_{21}[2]\oplus \Delta_{17}[2]\oplus [1]\oplus [15]$ \\
\hline
\small $\Delta_{19} [5]\oplus [14]$ & \small $\Delta_{19}[4]\oplus [1]\oplus [15]$ \\
\hline
\small $\Delta_{21} [3]\oplus \Delta_{17}\oplus \Delta_{15}\oplus [14]$ & \small $\Delta_{21}[2]\oplus \Delta_{17}[2]\oplus [1]\oplus [15]$ \\
\hline
\small $\Delta_{21} [3]\oplus \Delta_{15} [3]\oplus [12]$ & \small $\Delta_{21}[2]\oplus \Delta_{15}[4] \oplus [1]\oplus [11]$ \\
\hline
\small $\Delta_{23,13}\oplus \Delta_{19} [3]\oplus \Delta_{15}\oplus [12]$ & \small $\mathrm{Sym}^2\Delta_{11} \oplus \Delta_{19}[2]\oplus \Delta_{15}[2]\oplus [13]$ \\
\hline
\end{tabular} \caption{Paramètres standards $\psi(\pi',\mathrm{St})$ des éléments $\pi'\in \mathrm{Res}\, \pi$, pour les représentations $\pi\in \Pi_{\mathrm{cusp}}(\mathrm{SO}_{25})$ telles que $\pi_{\infty}$ est triviale, dont les valeurs propres associées pour $\mathrm{T}_2$ sont entières.}\label{res24} \end{bigcenter} \end{table}
\begin{table}[h!] \begin{bigcenter} \renewcommand\arraystretch{1.2}\begin{tabular}{|c|m{9cm}|}
\hline 
\rule[-0.3cm]{0mm}{0.8cm} $\psi (\pi,\mathrm{St})$ & \centering $\psi (\pi',\mathrm{St})$ \tabularnewline
\hline
$\Delta_{23}^2 \oplus [22]$ & \footnotesize \centering $[23]\oplus [1],\ \mathrm{Sym}^2\Delta_{11}\oplus [21]$ \tabularnewline
\hline
$\Delta_{23}^2 \oplus \Delta_{21} \oplus [20]$ & \footnotesize \centering $\mathrm{Sym}^2\Delta_{11}\oplus [21],\ \Delta_{21}[2]\oplus [1]\oplus [19]$ \tabularnewline
\hline
$\Delta_{23}^2 \oplus \Delta_{21} \oplus \Delta_{19} \oplus [18]$ & \footnotesize \centering $ \Delta_{21}[2]\oplus [1]\oplus [19],\ \mathrm{Sym}^2\Delta_{11}\oplus \Delta_{19}[2]\oplus [17] $ \tabularnewline
\hline
$\Delta_{23}^2 \oplus \Delta_{19} [3] \oplus [16]$ & \footnotesize \centering $ \mathrm{Sym}^2\Delta_{11}\oplus \Delta_{19}[2]\oplus [17],\ \Delta_{19}[4]\oplus [1]\oplus [15] $ \tabularnewline
\hline
$\Delta_{23}^2 \oplus \Delta_{21} \oplus \Delta_{19} \oplus \Delta_{17} \oplus [16]$ & \footnotesize \centering $\mathrm{Sym}^2\Delta_{11}\oplus \Delta_{19}[2]\oplus [17],\ \Delta_{21}[2]\oplus \Delta_{17}[2]\oplus [1]\oplus [15] $ \tabularnewline
\hline
$\Delta_{23}^2 \oplus \Delta_{19} [3] \oplus \Delta_{15} \oplus [14]$ & \footnotesize \centering $\Delta_{19}[4]\oplus [1]\oplus [15],\ \mathrm{Sym}^2\Delta_{11} \oplus \Delta_{19}[2]\oplus \Delta_{15}[2]\oplus [13]$ \tabularnewline
\hline
$\Delta_{23}^2 \oplus \Delta_{21} \oplus \Delta_{19} \oplus \Delta_{17} \oplus \Delta_{15} \oplus [14]$ & \footnotesize \centering $\Delta_{21}[2]\oplus \Delta_{17}[2]\oplus [1]\oplus [15],\ \mathrm{Sym}^2\Delta_{11} \oplus \Delta_{19}[2]\oplus \Delta_{15}[2]\oplus [13]$ \tabularnewline
\hline
\end{tabular} \caption{Paramètres standards $\psi(\pi',\mathrm{St})$ des éléments $\pi'\in \mathrm{Res}\, \pi$, pour les représentations $\pi\in \Pi_{\mathrm{cusp}}(\mathrm{SO}_{25})$ telles que $\pi_{\infty}$ est triviale, dont les valeurs propres associées pour $\mathrm{T}_2$ ne sont pas entières.}\label{res242} \end{bigcenter} \end{table}

\end{document}